\pdfoutput=1
\documentclass[a4paper,12pt,reqno,oneside]{amsart}

\usepackage{geometry}
\usepackage{lmodern}
\usepackage[stretch=10]{microtype}
\usepackage{graphicx}
\usepackage{tikz-cd}

\usepackage{amsmath}
\usepackage{amssymb}
\usepackage{amsthm}
\usepackage[matrix,arrow,cmtip]{xy}
\usepackage{bbm}
\usepackage[shortlabels]{enumitem}
\PassOptionsToPackage{hyphens}{url}
\usepackage[bookmarksnumbered,bookmarksdepth=2,bookmarksopen,colorlinks,linkcolor=black,citecolor=black,urlcolor=black,linktoc=all]{hyperref}
\usepackage[capitalise]{cleveref}
\usepackage{mathrsfs}

\DeclareFontFamily{U}{wncy}{}
\DeclareFontShape{U}{wncy}{m}{n}{<->wncyr10}{}
\DeclareSymbolFont{mcy}{U}{wncy}{m}{n}
\DeclareMathSymbol{\Sha}{\mathord}{mcy}{"58}

\newcommand{\Qbar}{\overline \bQ}

\newcommand{\bQ}{\mathbb{Q}}

\newcommand{\ZZ}{\mathbb{Z}}
\newcommand{\QQ}{\mathbb{Q}}
\newcommand{\RR}{\mathbb{R}}
\newcommand{\CC}{\mathbb{C}}

\newcommand{\GG}{\mathbb{G}}
\newcommand{\PP}{\mathbb{P}}
\newcommand{\AAA}{\mathbb{A}}

\newcommand{\cA}{\mathcal{A}}
\newcommand{\cB}{\mathcal{B}}

\newcommand{\cG}{\mathcal{G}}

\newcommand{\cM}{\mathcal{M}}

\newcommand{\cO}{\mathcal{O}}

\newcommand{\fP}{\mathfrak{P}}
\newcommand{\fQ}{\mathfrak{Q}}

\newcommand{\rM}{\mathrm{M}}

\newcommand{\rd}{\mathrm{d}}

\DeclareMathOperator{\denom}{denom}

\DeclareMathOperator{\Hom}{Hom}

\DeclareMathOperator{\lcm}{lcm}

\DeclareMathOperator{\rk}{rk}

\DeclareMathOperator{\Sin}{Sin}

\DeclareMathOperator{\spmap}{sp}

\DeclareMathOperator{\trdeg}{trdeg}

\newcommand{\van}{{v\textrm{-}\mathrm{an}}}
\newcommand\vhatan{{\hat v\textrm{-}\mathrm{an}}}

\newcommand{\wan}{{w\textrm{-}\mathrm{an}}}
\newcommand\whatan{{\hat w\textrm{-}\mathrm{an}}}

\newcommand{\abs}[1]{\lvert #1 \rvert}

\newcommand{\powerseries}[2]{#1 [\![ #2 ]\!]}

\newcommand{\ov}{\overline}

\newcommand{\defterm}[1]{\textbf{#1}}

\newcommand{\numfuns}{\mu}
\newcommand{\numfunsabstract}{\mu}
\newcommand{\nummonomials}{\lambda}
\newcommand{\numlinmons}{\mu}

\newcommand{\relation}{P}
\newcommand{\relfactor}{Q}

\newcommand{\sfrac}[2]{#1 / #2}

\newtheorem{lemma}{Lemma}[section]
\AddToHook{env/lemma/begin}{\crefalias{lemma}{lemma}}
\newtheorem{proposition}[lemma]{Proposition}
\AddToHook{env/proposition/begin}{\crefalias{lemma}{proposition}}
\newtheorem{theorem}[lemma]{Theorem}
\AddToHook{env/theorem/begin}{\crefalias{lemma}{theorem}}
\newtheorem{corollary}[lemma]{Corollary}
\AddToHook{env/corollary/begin}{\crefalias{lemma}{corollary}}

\AddToHook{env/conjecture/begin}{\crefalias{lemma}{conjecture}}
\Crefname{conjecture}{Conjecture}{Conjectures} 

\Crefname{claim}{Claim}{Claims}
\newtheorem*{lemma*}{Lemma}
\newtheorem*{proposition*}{Proposition}
\newtheorem*{theorem*}{Theorem}
\newtheorem*{corollary*}{Corollary}
\newtheorem*{claim*}{Claim}

\theoremstyle{definition}
\newtheorem*{definition}{Definition}
\newtheorem{remark}[lemma]{Remark}
\AddToHook{env/remark/begin}{\crefalias{lemma}{remark}}

\usepackage{ifthen}
\newcounter{constant}
\newcommand{\createC}[1]{\refstepcounter{constant} \label{C:#1}}
\newcommand{\newC}[1]{%
  \ifthenelse{\equal{#1}{*}} {%
      \stepcounter{constant} c_{\theconstant}%
  } {%
      \createC{#1} c_{\theconstant}%
  }%
}
\newcommand{\refC}[1]{c_{\ref*{C:#1}}}

\usepackage{color}

\newcommand{\martin}[1]{}

\title[Explicit height bounds for G-functions and tori]{Explicit height bounds for G-functions and unlikely intersections with lines in tori}

\author{Martin Orr}

\address{Orr: Department of Mathematics, The University of Manchester, Alan Turing Building, Oxford Road, Manchester M13 9PL, United Kingdom}
\email{martin.orr@manchester.ac.uk}


\begin{document}

\begin{abstract}
This paper computes explicit constants in Bombieri and Andr\'e's height bound for points at which there are unexpected ``global'' polymomial relations between values of G-functions.
It applies these bounds for G-functions to obtain explicit height bounds for unlikely intersections with lines in tori, making explicit a weak version of the bounded height theorem of Bombieri, Masser and Zannier and, in some cases, improving an explicit bound of Habegger.
\end{abstract}

\maketitle

\section{Introduction}

A G-function is a power series in $\powerseries{\Qbar}{X}$ satisfying a differential equation and a growth condition on its coefficients (for all absolute values on~$\Qbar$).
They were introduced by Siegel to study transcendence of numbers \cite{SFZ14}.
One of the main results on polynomial relations between values of G-functions is Bombieri's principle of global relations, sometimes known as the ``Hasse principle for values of G-functions'' \cite{Bom81}.

Andr\'e applied Bombieri's theorem to prove a height bound for what would later become known as unlikely intersections (in a situation involving endomorphisms of abelian varieties) \cite{And89}.
In recent years, Andr\'e's method of applying the principle of global relations has become an important tool in the study of unlikely intersections, being used to prove results on ``large Galois orbits'' and cases of the Zilber--Pink conjecture.
This includes work by Binyamini and Masser \cite{BM21}, Daw and the author of the present paper \cite{DO:ExCM}, \cite{DO:Y1n}, Papas \cite{Papas:ZP}, \cite{Papas:vadic-I} and Urbanik \cite{Urb25}, \cite{Urb:surfaces}, among other papers by the same authors.
For a summary of the method and applications, see Andr\'e's recent survey~\cite{And25}.

The primary goal of this paper is to establish explicit versions of the height bound in the principle of global relations.
As an illustration of the usefulness of these explicit height bounds, we derive an explicit bound for a weakened version of Bombieri, Masser and Zannier's theorem on unlikely intersections in tori \cite{BMZ99}.

While the G-functions method for unlikely intersections has often been described as ``effective in principle'', it turns out that non-trivial work is required to compute the input data for the explicit principle of global relations in terms of more natural input data.
This paper's methods could also be used to obtain explicit height bounds in applications of G-functions to unlikely intersections in Shimura varieties, although further work would be needed to compute these inputs in Shimura varieties settings.

\subsection{Height bound for relations between G-functions} \label{ssec:intro:g-functions}

Bombieri's principle of global relations deals with ``non-trivial global'' relations between values of G-functions.
Precise definitions of these notions involve some subtleties (see section~\ref{sec:relations}) but, essentially,
a ``global relation'' between G-functions $y_1,\dotsc,y_\mu$ at $\xi \in \Qbar$ is a polynomial~$P$ with coefficients in~$\Qbar$ such that the equation
\[ P(y_1(\xi), \dotsc, y_\mu(\xi)) = 0 \]
holds $v$-adically, for every place~$v$ at which this equation makes sense (that is, whenever $\xi$ lies inside the $v$-adic radii of convergence of $y_1, \dotsc, y_\mu$).
We say that $P$ is ``non-trivial'' if it does not come from a polynomial relation between the G-functions themselves.

Incorporating some improvements by Andr\'e, Bombieri's principle of global relations is as follows.

\begin{theorem} \label{andre-thm-E}
(\cite[Thm.~E]{And89}, building on \cite[Thm.~4]{Bom81})
Let $y_1,\dotsc,y_\numfuns$ be G-functions.
All points $\xi \in \Qbar$ at which there exists a non-trivial global polynomial relation~$\relation$ between the evaluations at~$\xi$ of $y_1,\dotsc,y_\numfuns$ satisfy
\[ h(\xi) < \refC{andre-thm-E-mult} \deg(\relation)^{\refC{andre-thm-E-exp}}, \]
where $\newC{andre-thm-E-mult}$ and~$\newC{andre-thm-E-exp}$ depend only on the G-functions $y_1,\dotsc,y_\numfuns$.
\end{theorem}

\cite[Thm.~4]{Bom81} and \cite[VII, 5.2, Thm.]{And89}\ give an explicit value for $\refC{andre-thm-E-exp}$.
\cite[Thm.~3]{Bom81} gives an explicit value for $\refC{andre-thm-E-mult}$ in the case $\deg(\relation)=1$, subject to finitely many exceptions (explicitly bounded in number).
Neither \cite{Bom81} nor \cite{And89} gives an explicit value for $\refC{andre-thm-E-mult}$ in general, although they give explicit bounds for some intermediate steps (\cite[Main~Thm.]{Bom81} and \cite[VII, 3.5, Prop.]{And89}).

This paper proves several explicit versions of \cref{andre-thm-E}:
\begin{enumerate}
\item \cref{andre-5.2-explicit}: an explicit value for $\refC{andre-thm-E-mult}$, with the same~$\refC{andre-thm-E-exp}$ as in \cite[Thm.~4]{Bom81} and~\cite[VII, 5.2, Thm.]{And89};
\item \cref{andre-bombieri} and \cref{andre-bombieri-simp}: more refined bounds which take into account additional information about the relations between $y_1,\dotsc,y_\numfuns$ at~$\xi$;
\item \cref{andre-bombieri-strong-general}: an explicit value for $\refC{andre-thm-E-mult}$, alongside a better value for~$\refC{andre-thm-E-exp}$, when the relation is ``super-strongly non-trivial'' (this improvement to~$\refC{andre-thm-E-exp}$, under a slightly weaker condition called ``strongly non-trivial'', was already included in \cite[VII, 5.2, Thm.]{And89}, building on \cite[Thm. 7 and~8]{Bom81});
\item \cref{andre-bombieri-HF,andre-bombieri-HF-bounds}: more refined bounds for super-strongly non-trivial relations, depending on how much information is available about the ideal of functional relations between the G-functions;
\item \cref{andre-bombieri-homog-alg-ind}: an explicit bound in the special case where the G-functions are homogeneously algebraically independent, as in the application in the second part of this paper.
\end{enumerate}

One could combine the refinements of (2) and~(3), to obtain still sharper bounds when both are applicable.
This combination of refinements is not included in this paper because it is not clear how to formulate a general statement.

\subsection{Application to unlikely intersections in tori} \label{ssec:intro:application}


Let $n \geq 2$.
For $d \in \ZZ_{\geq 1}$, let $\Sigma_d$ denote the union of all algebraic subgroups of~$\GG_m^n$ of codimension at least~$d$.
This paper's main theorem on unlikely intersections in~$\GG_m^n$ is the following.

\begin{theorem} \label{lines-in-torus-height-bound-intro}
Let $C \subset \GG_m^n$ be an algebraic curve whose Zariski closure in~$\AAA^n$ is a line defined over a number field~$K$.
Suppose that $C$ contains the point $s_0 := (1,1,\dotsc,1)$, and that $C$ is not contained in a proper algebraic subgroup of~$\GG_m^n$.

Then all points $s \in C(\ov K) \cap \Sigma_2$ satisfy
\[ h(s) \leq \refC{lines-in-torus-height-bound-intro-mult} [K(s):\QQ]^n \bigl( \log([K(s):\QQ])+1 \bigr), \]
where $\newC{lines-in-torus-height-bound-intro-mult}$ is given by an explicit formula (\cref{lines-in-torus-height-bound}) in terms of~$n$ and a suitable measure of the height of~$C$.
\end{theorem}

Calculating the constant in \cref{lines-in-torus-height-bound-intro} in the case~$n=3$ yields the following.

\begin{corollary} \label{lines-in-Gm3-height-bound-intro}
Let $K$ be a number field.
Let $a_1,a_2,a_3 \in K^\times$ be distinct.
Let 
\[ H = h(a_1) + h(a_2) + h(a_3). \]
For all $\xi \in \Qbar$, if $1+a_1\xi$, $1+a_2\xi$ and $1+a_3\xi$ lie in a rank-$1$ subgroup of $\GG_m(\Qbar)$, then
\begin{enumerate}[(i)]
\item $h(\xi) < 2708(H+1)$ if $K(\xi)$ is totally real or $[K(\xi):\QQ] \leq 3$;
\item $h(\xi) < 166 [K(\xi):\QQ]^3 \bigl( \log([K(\xi):\QQ]) + 2.5 \bigr) (H+1)$ otherwise.
\end{enumerate}
\end{corollary}




In order to compute the values for the constants in \cref{lines-in-Gm3-height-bound-intro}, it is necessary to compute the input data for \cref{andre-bombieri,andre-bombieri-homog-alg-ind} in this situation.
This requires investigation of the differential equation satisfied by the G-functions, their radii of convergence and the radii inside which the global relations are valid.

Even if we ignore the explicit constants, this is the first written application of the G-functions method to unlikely intersections in algebraic group, as opposed to Shimura varieties or variations of Hodge structure.
The application to tori is simpler than the previous applications to Shimura varieties.
Therefore, the proof in section~\ref{ssec:tori-inexplicit} that there exists a bound of the form stated in \cref{lines-in-torus-height-bound-intro}, without calculating explicit constants, should be of interest as an illustration of the G-functions method for unlikely intersections in its simplest possible case.

\subsection{Previous results on intersections in tori} \label{ssec:intro:tori-previous}

The first theorem on height bounds for unlikely intersections in tori was the following theorem of Bombieri, Masser and Zannier.
It also includes ``just likely'' intersections.

\begin{theorem} \label{bmz} \cite[Thm.~1]{BMZ99}
Let $C$ be an irreducible algebraic curve in~$\GG_m^n$, defined over~$\Qbar$.
Suppose that $C$ is not contained in a translate of a proper algebraic subgroup of~$\GG_m^n$.
Then $C \cap \Sigma_1$ is a set of $\Qbar$-points of bounded height.
\end{theorem}

The proofs of \cref{bmz} in \cite{BMZ99} are in principle effective.
An explicit bound was obtained by Habegger \cite[Thm.~11]{Hab17} (\cref{habegger-bound}).

\subsubsection{Comparison of Theorem~\ref{lines-in-torus-height-bound-intro} with Theorem~\ref{bmz}} \label{torus-bound-restrictions}

\Cref{lines-in-torus-height-bound-intro} is weaker than \cref{bmz} and \cite[Thm.~11]{Hab17} in the following ways:
\begin{enumerate}
\item The height bound in \cref{lines-in-torus-height-bound-intro} depends (polynomially) on $[K(s):\QQ]$, while the height bounds in \cref{bmz} and \cite[Thm.~11]{Hab17} are independent of the field of definition of~$s$.
\item \Cref{lines-in-torus-height-bound-intro} only bounds the height of points in $C \cap \Sigma_2$ (unlikely intersections), instead of $C \cap \Sigma_1$ (unlikely and just likely intersections).
\item In \cref{lines-in-torus-height-bound-intro}, the curve~$C$ must contain the identity point $(1,1,\dotsc,1)$ of~$\GG_m^n$.
(This could be weakened to saying that $C$ must contain a torsion point of~$\GG_m^n$, but that does not gain anything significant.)
\item In \cref{lines-in-torus-height-bound-intro}, the closure in~$\AAA^n$ of~$C$ must be a line.
\end{enumerate}

Restrictions (1), (2) and~(3) appear to be essential to the G-functions method for intersections in algebraic groups, even if one does not ask for explicit bounds.
Restriction~(1) arises because we build a global relation as a product of $v$-adic relations, involving a factor for each complex place of~$K(s)$, so that the degree of the global relation may be proportional to $[K(s):\QQ]$.
Restrictions (2) and~(3) are needed to construct $v$-adic relations at archimedean places with coefficients in~$\Qbar$ -- without these restrictions, we could only construct relations with coefficients involving $2\pi i$ or logarithms of coordinates of a chosen base point on~$C$.

Restriction~(4) is not essential to the G-functions method, but it significantly simplifies the argument and leads to stronger explicit bounds because it avoids the need to pass to a cover of~$C$ with a ``good local parameter'', as in \cite[Lemma~5.1]{DO:Y1n}.
Restriction~(4) will be removed in a later paper, which will also apply the same method to abelian varieties.

\subsubsection{Comparison of Theorem~\ref{lines-in-torus-height-bound-intro} with \cite[Thm.~11]{Hab17}}

Restriction~(1) means that the explicit height bounds coming from \cref{lines-in-torus-height-bound-intro} are inevitably worse than those of \cite[Thm.~11]{Hab17} for intersection points of large Galois degree.
However, the bounds of \cref{lines-in-torus-height-bound-intro} are better for points of small Galois degree.
Indeed, in the case $n=3$, the bound in \cref{lines-in-torus-height-bound-intro} is better than \cite[Thm.~11]{Hab17} for $[K(s):\QQ]$ up to around~$10^{10^{12}}$, and this cross-over point increases for larger~$n$ (see section~\ref{ssec:habegger}).
Even if one optimises the arguments of \cite{Hab17} for the special case of lines, \cref{lines-in-torus-height-bound-intro} is better for $[K(s):\QQ]$ up to around~$10\,000$.

\subsubsection{Application to the Zilber--Pink conjecture}

It is well-known that \cref{bmz} implies the Zilber--Pink conjecture for curves in~$\GG_m^n$, that is:

\begin{theorem} \label{zp-torus} \cite[Thm.~2]{BMZ99}
In the setting of \cref{bmz}, $C \cap \Sigma_2$ is finite.
\end{theorem}

Restrictions (1) and~(2) do not pose a problem for deducing \cref{zp-torus} from \cref{bmz}.
Indeed, a height bound which is polynomial in the degree of the point suffices to obtain the Galois bound needed to prove \cref{zp-torus} using the Pila--Zannier method (as described in \cite{CMPZ16}), although it does not appear to be sufficient to deduce \cref{zp-torus} using the argument of \cite[sec.~4]{BMZ99}.
\Cref{zp-torus} is about $C \cap \Sigma_2$ (and its generalisation to $C \cap \Sigma_1$ is false) so restriction~(2) does not matter for \cref{zp-torus}.
Thus, one could in principle deduce \cref{zp-torus}, subject to restrictions (3) and~(4), from \cref{lines-in-torus-height-bound-intro}.

\subsection*{Acknowledgements}

The author thanks Y. André, I. Cruciani, C. Daw, N. Dogra, G. Fowler, P. Habegger, G. Jones, R. Richard and D. Urbanik for useful discussions around the work presented in this paper.
This project was supported by the Engineering and Physical Sciences Research Council (EP/Y020758/1).

\section{Definitions: Heights and sizes, differential operators and G-functions} \label{sec:sizes}

In this section, we define notation and normalisations for various quantities related to absolute values and heights, as well as some notions around differential operators and G-functions.

\subsection{Absolute values and heights}

Let $K$ be a number field.

For each place $v$ of~$K$, let $K_v$ denote the $v$-adic completion of~$K$.
Let $p_v \in \ZZ_{>0} \cup \{ \infty \}$ denote the characteristic of the residue field if $v$ is non-archimedean, or $\infty$ if~$v$ is archimedean.
Let $\QQ_v = \QQ_{p_v}$ if $v$ is non-archimedean, or~$\RR$ if $v$ is non-archimedean.
Let $\abs{\cdot}_v$ denote the $v$-adic absolute value on~$K$, normalised so that it extends the usual $p_v$-adic absolute value on~$\QQ$.

Heights and related quantities are defined by summing over all places of~$K$.
With our normalisation of absolute values, the sum needs to be weighted by the local degrees.
In order to write this conveniently, we define the following notation.

\begin{definition}
If $V$ is a subset of the places of~$K$ and $a_v$ are real numbers indexed by~$V$, almost all zero, then write
\[ {\sum_{v \in V}}' a_v := \sum_{v \in V} \frac{[K_v:\QQ_v]}{[K:\QQ]} a_v. \]
We also write $\sum_v'$ without specifying the set~$V$ to denote the sum over all places of~$K$, still weighted by $[K_v:\QQ_v]/[K:\QQ]$.
\end{definition}

This paper normalises absolute values differently from \cite{And89}. As a result, the local quantities below are normalised differently from in \cite{And89}.
By using $\sum_v'$ with our normalisation instead of~$\sum_v$, we obtain the same global quantities as \cite{And89}.

\medskip

For $x \in \RR$, write $\log^+(x) =\log \max \{ 1, x \}$.

For $\alpha \in K$, let $h(\alpha)$ denote the absolute (logarithmic) Weil height:
\[ h(\alpha) = {\sum_v}' \log^+(\abs{\alpha}_v). \]
For $\alpha \in \Qbar$, $h(\alpha)$ is independent of the choice of number field~$K$ containing~$\alpha$.

Let $\Xi$ be a finite subset of~$K \setminus \{0\}$.
We define:
\begin{enumerate}
\item for each place $v$ of~$K$, $\sigma_v(\Xi) = \min \bigl( \{ 1 \} \cup \{ \abs{\alpha}_v : \alpha \in \Xi \} \bigr)$;
\item $\sigma(\Xi) = \sum_v' \log(1/\sigma_v(\Xi))$;
\item $H(\Xi) = \sum_{\alpha \in \Xi} h(\alpha)$.
\end{enumerate}
Observe that $\sigma(\Xi)$ is the usual (affine, logarithmic, absolute) Weil height of the vector formed by the elements of~$\Xi$.
The quantities $\sigma(\Xi)$ and $H(\Xi)$ depend only on~$\Xi \subset \Qbar \setminus \{0\}$, not on the number field~$K$ containing~$\Xi$.
Clearly,
\begin{equation} \label{eqn:sigma-H-comparison}
\sigma(\Xi) \leq H(\Xi) \leq \#\Xi \cdot \sigma(\Xi).
\end{equation}

\subsection{Power series}


Let $y_1, \dotsc, y_\numfuns \in \powerseries{K}{X}$, and write $y_i = \sum_{j=0}^\infty y_{ij}X^j$.
Following \cite{And89}, but using this paper's normalisation of absolute values, define:
\begin{enumerate}
\item for each place~$v$ of~$K$ and each $n \in \ZZ_{>0}$,
\[ h_{v,n}(y_1,\dotsc,y_\numfuns) = \frac{1}{n} \max \bigl\{ \log^+(\abs{y_{ij}}_v) : 1 \leq i \leq \numfuns, \, 0 \leq j \leq n \bigr\}; \]
\item for each place $v$ of~$K$, the \defterm{$v$-adic radius of convergence} of~$y_1,\dotsc,y_\numfuns$ is 
\[ R_v(y_1,\dotsc,y_\numfuns) = \min_{i=1,\dotsc,\numfuns} \liminf_{j \geq 1} \, \abs{y_{ij}}_v^{-1/j}; \]
\item the \defterm{global radius} of $y_1,\dotsc,y_\numfuns$ is
\[ \rho(y_1,\dotsc,y_\numfuns) = {\sum_v}' \log^+(1/R_v(y_1,\dotsc,y_\numfuns)) = {\sum_v}' \limsup_{n \geq 1} h_{v,n}(y_1,\dotsc,y_\numfuns); \] 
\item the \defterm{size} of $y_1,\dotsc,y_\numfuns$ is 
\[ \sigma(y_1,\dotsc,y_\numfuns) = \limsup_{n \geq 1} {\sum_v}' h_{v,n}(y_1,\dotsc,y_\numfuns). \] 
\end{enumerate}

Observe that $\rho(y_1,\dotsc,y_\numfuns)$ and $\sigma(y_1,\dotsc,y_\numfuns)$ are independent of the field~$K$ containing the coefficients of $y_1,\dotsc,y_\numfuns$.

We write $y_i^\van$ for the analytic function on the disc $D(0,R_v(y_i),K_v)$ induced by~$y_i$ via the tautological embedding $K \to K_v$.

\subsection{Differential operators and \texorpdfstring{$\partial$}{d}-modules}




In this article, a \defterm{differential operator} over a field~$K$ means an operator of the form
\[ \Lambda = \frac{d}{dX} - \Gamma, \]
where $\Gamma \in \rM_\numfunsabstract(K(X))$.


Let $A$ be a $K[X]$-algebra equipped with a derivation extending $d/dX$ on~$K[X]$.
(In practice, $A$ will be a ring of power series or Laurent series in~$X$, and the derivation on~$A$ will be the one ``naturally'' denoted $\frac{d}{dX}$.)
A \defterm{solution} of a differential operator $\Lambda = d/dX - \Gamma$ in~$A$ means a vector $\underline y \in A^\numfunsabstract$ such that
\[ \frac{d}{dX} \underline y = \Gamma \underline y. \]


It is sometimes useful to use the ``basis-independent'' language of $\partial$-modules.
Let $\partial$ denote the derivation $X \frac{d}{dX}$ on $K[X]$.
A \defterm{$\partial$-module} means a $K[X]$-module~$\cM$ equipped with an endomorphism, again denoted~$\partial$, satisfying the Leibniz property
\[ \partial(fy) = \partial(f)y + f\partial(y) \]
for all $f \in K[X]$ and $y \in \cM$.
(In this article, all the $\partial$-modules we will consider will be $K(X)$-vector spaces.
Since $X$ is invertible in~$K(X)$, there is an obvious equivalence between such $\partial$-modules and differential modules over $(K(X), d/dX)$.
We follow \cite{And89} in calling them $\partial$-modules instead of differential modules.)


Every differential operator~$\Lambda = d/dX - \Gamma$ gives rise to a $\partial$-module structure on~$K(X)^\mu$, where the action of~$\partial$ on the standard basis~$\{m_1,\dotsc,m_\numfunsabstract\}$ is defined by
\begin{equation} \label{eqn:partial-Gamma}
\partial m_j = \sum_{i=1}^\numfunsabstract X \Gamma_{ij} m_i.
\end{equation}
Conversely, given a $\partial$-module which is a $K(X)$-vector space, together with a $K(X)$-basis~$\{m_1,\dotsc,m_\numfunsabstract\}$, we obtain a differential operator~$\Lambda = d/dX - \Gamma$ by defining the entries of~$\Gamma$ via equation~\eqref{eqn:partial-Gamma}.

\subsection{Singularities of differential operators}

Let $\Lambda = d/dX - \Gamma$ be a differential operator over a field~$K$, where $\Gamma \in \rM_\mu(K(X))$.
An \defterm{ordinary point} of~$\Lambda$ is a point in~$\ov K$ where no entry of~$\Gamma$ has a pole.
A \defterm{finite singularity} of~$\Lambda$ is a point in~$\ov K$ where at least one entry of~$\Gamma$ has a pole.
An \defterm{apparent singularity} of~$\Lambda$ is a finite singularity $a \in \ov K$ such that the $\ov K$-vector space of solutions of~$\Lambda$ in~$\powerseries{\ov K}{X-a}$ has dimension~$\numfunsabstract$.
We write $\Sin(\Lambda)$ for the set of non-apparent finite singularities of~$\Lambda$.

\begin{remark}
Over~$\CC$, it is common to define (non-)apparent singularities in terms of holomorphic solutions of~$\Lambda$ at~$a$ rather than formal power series solutions.
These two definitions are equivalent, thanks to \cite[Thm.~2]{BM15} and Cauchy's theorem on the existence of solutions at an ordinary point of a differential equation.
\end{remark}

A \defterm{good apparent singularity} of~$\Lambda$ is an apparent singularity~$a \in \ov K$ such that the ``evaluation at~$a$'' map 
\[ \bigl\{ \text{solutions  of~$\Lambda$ in $\powerseries{\ov K}{X-a}$} \bigr\} \to \ov K^\numfunsabstract \]
is injective, hence bijective.
(This map is also bijective for ordinary points.)

\subsection{Size and radii of differential operators}

Let $K$ be a number field and let $\Lambda = d/dX - \Gamma$ be a differential operator over~$K$.
Define $\Gamma_{[j]} \in \rM_\numfunsabstract(K(X))$, for $j \in \ZZ_{\geq0}$, by the recurrence relation
\[ \Gamma_{[0]} = I, \quad \Gamma_{[j+1]} = \Gamma_{[j]}\Gamma + \frac{d}{dX} \Gamma_{[j]}. \]
These matrices have the property that, if $\underline y$ is a solution of~$\Lambda$, then
\[ \frac{d^j}{dX^j} \underline y = \Gamma_{[j]} \cdot \underline y \]
for all~$j$.
Note that $\Gamma_{[j]} = X^{-j}G_{[j]}$, where $G_{[j]}$ is defined by \cite[III, 1.4, (8)]{And89}.

Following \cite[IV]{And89}, we define:
\begin{enumerate}
\item for each place~$v$ of~$K$ and each $n \in \ZZ_{>0}$,
\[ h_{v,n}(\Lambda) =
\begin{cases}
    0 & \text {if $v$ is archimedean,}
\\ \frac{1}{n} \log^+ \max_{j \leq n} \lVert \Gamma_{[j]}/j! \rVert_v & \text{if $v$ is non-archimedean,}
\end{cases} \]
where $\lVert M \rVert_v$ denotes the maximum of the $v$-adic Gauss absolute values of the entries of~$M \in \rM_\numfunsabstract(K(X))$; 
\item for each place $v$ of~$K$,
\[ R_v(\Lambda) =
\begin{cases}
    1 & \text {if $v$ is archimedean,}
\\ \sup \{ r \leq 1 : \Lambda \text{ has a solution in } \cA(t_v,r) \} & \text{if $v$ is non-archimedean,}
\end{cases} \]
where $\cA(t_v,r)$ is the ring of analytic functions on the disc $D(t_v,r)$ around a $v$-adic generic point~$t_v$ satisfying $\abs{t_v - a}_v = 1$ for all $a \in D(0, 1, K_v)$; 
\item $\rho(\Lambda) = \sum_v' \log^+(1/R_v(\Lambda)) = \sum_v' \limsup_n h_{v,n}(\Lambda)$; 
\item $\sigma(\Lambda) = \limsup_n \sum_v' h_{v,n}(\Lambda)$. 
\end{enumerate}

As observed in \cite[IV, 3.3 and 4.1]{And89}, $R_v(\Lambda)$, $\rho(\Lambda)$ and $\sigma(\Lambda)$ depend only on the $\partial$-module~$\cM$ corresponding to~$\Lambda$, not on the choice of basis of~$\cM$.
We may denote them by $R_v(\cM)$, $\rho(\cM)$, $\sigma(\cM)$ respectively.
Also, $\rho(\Lambda)$ and $\sigma(\Lambda)$ are invariant under extension of the base field~$K$.

\subsection{G-functions} \label{ssec:g-functions-definition}

\begin{definition} \cite[p.~1]{And89}
A \defterm{G-function} is a power series $y(X) = \sum_{n \geq 0} a_nX^n$ with coefficients $a_n$ in a number field $K$,  which satisfies the following conditions:
\begin{enumerate}
\item there exists $\newC{G-function-arch-base}>0$ such that $\abs{a_n} < \refC{G-function-arch-base}^n$ for all $n$ and for all archimedean absolute values $\abs{\cdot}$ of~$K$;
\item there exists a sequence of positive integers $(d_n)$ which grows at most geometrically such that $d_na_m$ is an algebraic integer for all $m \leq n$;
\item $y(X)$ satisfies a linear homogeneous differential equation
\[ \frac{d^\numfunsabstract}{dX^\numfunsabstract}y + \gamma_{\numfunsabstract-1} \frac{d^{\numfunsabstract-1}}{dX^{\numfunsabstract-1}}y + \dotsb + \gamma_1 \frac{d}{dX}y + \gamma_0y = 0 \]
with coefficients $\gamma_i \in K(X)$.
\end{enumerate}
\end{definition}

Equivalently, a G-function is a power series $y \in \powerseries{\Qbar}{X}$ satisfying $\sigma(y) < \infty$ which solves a linear homogeneous differential equation over~$\Qbar(X)$ \cite[I, 1.3]{And89}.

\begin{remark}
The G-functions used in section~\ref{sec:torus-application} of this paper are of the form $\log(1+aX)$, where $a \in K^\times$.
Note that
\[ \log(1+aX) = \sum_{n \geq 1} \frac{(-1)^{n+1}a^n}{n}X^n \]
has coefficients in~$K$.
It satisfies (1) with $\refC{G-function-arch-base} = 2 \max_v \abs{a}_v$ where $v$ runs over the archimedean places of~$K$, it satisfies (2) with 
$d_n = \lcm(1,\dotsc,n) \cdot \denom(a)^n$,
and it satisfies (3) since
\[ \frac{\rd^2}{\rd X^2} \log(1+aX) = \frac{-a^2}{(1+aX)^2} = \frac{-a}{1+aX} \cdot \frac{\rd}{\rd X} \log(1+aX). \]
\end{remark}

\section{Relations between evaluations of G-functions} \label{sec:relations}

In this section, we define various kinds of polynomial ``relations'' between evaluations of G-functions, and prove some properties of these notions.
Throughout the section, $y_1,\dotsc,y_\numfuns \in \powerseries{\Qbar}{X}$ denote power series whose coefficients generate a finite extension of~$\QQ$ (in practice, they will be G-functions), $\xi$ denotes an element of $\Qbar$, and $\relation$ denotes a homogeneous polynomial in $\Qbar[Y_1,\dotsc,Y_\numfuns]$.

\subsection{Functional and trivial relations}

We begin with a notion which appears in \cite{And89} and \cite{Urb25}, but is not given a name there.

\begin{definition}
Let $\tilde\relation \in \Qbar[X][Y_1,\dotsc,Y_\numfuns]$ be a polynomial which is homogeneous with respect to~$Y_1,\dotsc,Y_\numfuns$.
We say that $\tilde\relation$ is a \defterm{functional relation between $y_1, \dotsc, y_\numfuns$} if the following relation holds in $\powerseries{\Qbar}{X}$:
\[ \tilde \relation(X)(y_1(X), \dotsc, y_\numfuns(X)) = 0. \]
\end{definition}

The following definition of ``non-trivial'' is from \cite[VII, 5.1]{And89}.

\begin{definition}
Let $\relation \in \Qbar[Y_1,\dotsc,Y_\numfuns]$ be a homogeneous polynomial.
We say that:
\begin{enumerate}
\item $\relation$ is a \defterm{trivial relation between $y_1, \dotsc, y_\numfuns$ at~$\xi$} if it is the specialisation at $X=\xi$ of a functional relation (in $\Qbar[X][Y_1, \dotsc, Y_\numfuns]$) between $y_1, \dotsc, y_\numfuns$.
\item $\relation$ is \defterm{non-trivial as a relation between $y_1, \dotsc,y_\numfuns$ at~$\xi$} if it is not a trivial relation at~$\xi$.
\end{enumerate}
\end{definition}

The following definition is from \cite[\S 12, p.~63]{Bom81} and \cite[VII, 4.2]{And89}.

\begin{definition}
Let $\relation \in \Qbar[Y_1,\dotsc,Y_\numfuns]$ be a homogeneous polynomial.
Let $V \subset \PP^{\numfuns-1}_{\Qbar(X)}$ denote the variety cut out by the functional relations between $y_1,\dotsc,y_\numfuns$.
Let $W \subset \PP^{\numfuns-1}_{\Qbar}$ denote the variety cut out by the trivial relations at~$\xi$, together with~$\relation$.
We say that $\relation$ is \defterm{strongly non-trivial as a relation between $y_1,\dotsc,y_\numfuns$ at~$\xi$} if
\[ \dim_{\Qbar}(W) < \dim_{\Qbar(X)}(V). \]
\end{definition}

\begin{remark} \label{strongly-non-trivial-def-errors}
The definition of ``strongly non-trivial'' in \cite[VII, 5.1]{And89} appears to be erroneous. That definition is: ``$\relation$ does not occur as a factor of the specialisation at~$\xi$ of an irreducible functional relation between $y_1,\dotsc,y_\numfuns$''.
However, consider a situation in which the functional relations are generated by a single polynomial~$\tilde R$.
Then $\tilde R$ is the only irreducible functional relation between $y_1,\dotsc,y_\numfuns$ (up to a constant factor), so every homogeneous polynomial of degree greater than $\deg(\tilde R)$ would be strongly non-trivial by this definition, which is clearly not what is intended.

One could modify the definition in \cite[VII, 5.1]{And89} by dropping the word ``irreducible'', yielding:
``$\relation$ does not occur as a factor of the specialisation at~$\xi$ of a non-zero functional relation between $y_1,\dotsc,y_\numfuns$''.
This still does not do what one wants because, as long as there exists at least one non-zero functional relation~$\tilde R$, all homogeneous polynomials would be strongly non-trivial by this definition (each polynomial~$\relation$ would be a factor of the specialisation of the functional relation~~$\relation \tilde R$).
\end{remark}

The following characterisation of strongly non-trivial relations is sometimes more convenient than the definition above.

\begin{lemma} \label{strongly-non-trivial-equivalent-defs}
Let $\relation \in \Qbar[Y_1,\dotsc,Y_\numfuns]$ be a homogeneous polynomial.
Let $V_\xi \subset \PP^{\numfuns-1}_{\Qbar}$ denote the variety cut out by the trivial relations between $y_1,\dotsc,y_\numfuns$ at~$\xi$.
Then $\relation$ is strongly non-trivial at~$\xi$ if and only if it does not vanish identically on any irreducible component of~$V_\xi$.

(Equivalently, $\relation$ is strongly non-trivial at~$\xi$ if and only if it is not in any minimal prime ideal associated with the ideal generated by trivial relations at~$\xi$.)
\end{lemma}

\begin{proof}
Let $\tilde I \subset \Qbar[X][Y_1, \dotsc, Y_\numfuns]$ denote the ideal generated by functional relations.
Let $\tilde V \subset \AAA^1_{\Qbar} \times \PP^{\numfuns-1}_{\Qbar}$ denote the closed subscheme cut out by~$\tilde I$, where $\AAA^1$ corresponds to the indeterminate $X$ and $\PP^{\numfuns-1}$ to~$Y_1,\dotsc,Y_\numfuns$.
Now, $V$ (from the definition of strongly non-trivial relation) and $V_\xi$ (from the statement of the lemma) are the fibres of the projection $\tilde V \to \AAA^1_{\Qbar}$ over the generic point and over~$\xi$, respectively.

By \cref{functional-relations-ideal-is-prime} below, $\tilde V$ is integral.
Hence $\tilde V \to \AAA^1_{\Qbar}$ is flat, by \cite[Cor.~4.3.10]{Liu06}.
It follows that every irreducible component of $V_\xi$ has dimension $\dim(\tilde V) - \dim(\AAA^1)$, and the same is true for the generic fibre~$V$ \cite[Cor.~4.3.14]{Liu06}. 

Thus, every component of~$V_\xi$ has dimension (as a $\Qbar$-variety) equal to $\dim_{\Qbar(X)}(V)$.
It follows that $\dim_{\Qbar}(W) < \dim_{\Qbar(X)}(V)$ if and only if $W \cap C$ is strictly contained in~$C$, for every irreducible component $C$ of~$V_\xi$.
This is equivalent to saying that $\relation$ does not vanish identically on any irreducible component of~$V_\xi$, as claimed.
\end{proof}

The notion of ``strongly non-trivial relation'' is what Bombieri and André needed to obtain inexplicit height bounds, using the asymptotic behaviour of the Hilbert function.
In order to obtain explicit height bounds of the same type, this is not sufficient -- we require explicit bounds on individual values of the Hilbert function, as well as the following stronger condition on relations.

\begin{definition}
Let $\relation \in \Qbar[Y_1,\dotsc,Y_\numfuns]$ be a homogeneous polynomial.
We say that $\relation$ is \defterm{super-strongly non-trivial as a relation between $y_1,\dotsc,y_\numfuns$ at~$\xi$} if there does not exist any homogeneous polynomial $\relfactor \in \Qbar[Y_1,\dotsc,Y_\numfuns]$ such that $PQ$ is trivial at~$\xi$, but $\relfactor$ is non-trivial at~$\xi$.
\end{definition}

The next two lemmas show that ``super-strongly non-trivial'' is indeed strictly stronger than ``strongly non-trivial''.

\begin{lemma}
If $\relation$ is super-strongly non-trivial, then it is strongly non-trivial.
\end{lemma}

\begin{proof}
Let $I_\xi \subset \Qbar[Y_1,\dotsc,Y_\numfuns]$ denote the ideal generated by trivial relations at~$\xi$.
Let an irredundant primary decomposition of~$I_\xi$ be
\[ I_\xi = \bigcap_{i=1}^r \fP_i \cap \bigcap_{j=1}^s \fQ_j, \]
where $\sqrt{\mathstrut\fP_i}$ are minimal primes over~$I_\xi$, while $\sqrt{\fQ_j}$ are embedded primes.

\pagebreak 

For each $i=1,\dotsc,r$, use prime avoidance to choose $g_i$ such that:
\begin{enumerate}
\item $g_i \in \fP_i$,
\item $g_i \notin \sqrt{\mathstrut\fP_{i'}}$ for every $i' \neq i$,
\item $g_i \in \fQ_j$ for every~$j$ such that $\sqrt{\mathstrut\fP_i} \subset \sqrt{\fQ_j}$,
\item $g_i \notin \sqrt{\fQ_j}$ for every~$j$ not covered by~(3).
\end{enumerate}

Let
\[ g = \prod_{i : \relation \notin \sqrt{\mathstrut\fP_i}} g_i \]
and let $\relation' = \relation g$.
It follows from properties (1)--(4) of~$g_i$ that:
\begin{enumerate}
\item $g \in \fP_i$ for all $i$ such that $\relation \notin \sqrt{\mathstrut\fP_i}$;
\item $g \notin \sqrt{\mathstrut\fP_{i'}}$ for every $i'$ such that $\relation \in \sqrt{\mathstrut\fP_{i'}}$;
\item $g \in \fQ_j$ for every $j$ such that there exists an~$i$ satisfying $\relation \notin \sqrt{\mathstrut\fP_i} \subset \sqrt{\fQ_j}$;
\item $g \notin \sqrt{\fQ_j}$ for every~$j$ not covered by~(3).
\end{enumerate}

We have $\relation' \in \sqrt{\mathstrut\fP_i}$ for all~$i$, so $\relation' \in \sqrt{I_\xi}$.
Consequently, there exists $n \in \ZZ_{\geq 1}$ such that $(\relation')^n \in I_\xi$.
Choose the minimal such $n \in \ZZ_{\geq 1}$.
Let
\[ \relfactor = (\relation')^{n-1}g. \]
By the choice of~$n$, we have that $PQ \in I_\xi$.
Since $\relation$ is super-strongly non-trivial, it follows that $\relfactor \in I_\xi$.

Properties (1)--(4) of~$g$ imply that each~$\fP_i$ contains either $g$ or~$(\relation')^{n-1}$, and the same for each~$\fQ_j$.

If $n \geq 2$, then we have $g \mid \relation' \mid (\relation')^{n-1}$, so we obtain that $(\relation')^{n-1}$ is in every $\fP_i$ and every~$\fQ_j$.
It follows that $(\relation')^{n-1} \in I_\xi$, contradicting the minimality of~$n$.

Otherwise, $n=1$.
In this case, $(\relation')^{n-1}=1 \mid g$, so we obtain that $g$ is in every $\fP_i$.
Thanks to property~(3) of~$g$, it follows that $\relation \notin \sqrt{\mathstrut\fP_i}$ for every~$i$.
By \cref{strongly-non-trivial-equivalent-defs}, this implies that $\relation$ is strongly non-trivial.
\end{proof}

\begin{lemma}
Not every strongly non-trivial relation is super-strongly non-trivial.
\end{lemma}

\begin{proof}
The example will be given by the set of G-functions
\begin{align*}
    y_1 = 1,
\quad y_2 &= (1+X)^{3/4} \log(1+X),
\\ y_3 &= (1+X)^{1/4} \log(1+X)^3,
\quad y_4 = \log(1+X)^4.
\end{align*}

In order to describe the functional relations between these, consider first
\[ y_1' = 1+X, \; y_2, \; y_3, \; y_4. \]
Since $(1+X)^{1/4}$ and $\log(1+X)$ are algebraically independent, the ideal of $\Qbar$-relations between $y_1', y_2, y_3, y_4$ is the same as the ideal of the projective variety
\[ \{ [s^4 : s^3t : st^3 : t^4] \in \PP^3_{\Qbar} : [s,t] \in \PP^1 \}. \] 
This variety is a rational quartic curve, and its ideal is known to be generated by
\[ Y_1'Y_4 - Y_2Y_3, \quad Y_1'^2Y_3 - Y_2^3, \quad Y_2Y_4^2 - Y_3^3, \quad Y_1'Y_3^2 - Y_2^2Y_4. \]
To obtain generators for the ideal of $\Qbar[X]$-relations between $y_2$, $y_3$ and~$y_4$, simply substitute $1+X$ for $Y_1'$.
Since $y_1=1$, we obtain generators for the functional relations between $y_1$, $y_2$, $y_3$ and~$y_4$ by homogenising the resulting polynomials with respect to $Y_1,Y_2,Y_3,Y_4$, by inserting powers of~$Y_1$.
This yields
\[ (1+X)Y_1Y_4 - Y_2Y_3, \quad (1+X)^2Y_1^2Y_3 - Y_2^3, \quad Y_2Y_4^2 - Y_3^3, \quad (1+X)Y_1Y_3^2 - Y_2^2Y_4. \]

Consequently, the trivial relations between $y_1,y_2,y_3,y_4$ at~$-1$ are generated by
\begin{equation} \label{eqn:trivial-generators}
-Y_2Y_3, \quad -Y_2^3, \quad Y_2Y_4^2 - Y_3^3, \quad -Y_2^2Y_4.
\end{equation}
The projective variety cut out by~\eqref{eqn:trivial-generators} is $\{ [s:0:0:t] \} \subset \PP^3$.
This is irreducible, and~$Y_4$ does not vanish identically on it, so $Y_4$ is strongly non-trivial at~$-1$.

From~\eqref{eqn:trivial-generators}, $Y_2^2Y_4$ is a trivial relation at~$-1$.
However, thanks to~\eqref{eqn:trivial-generators}, the only trivial relations of degree~$2$ are the scalar multiples of~$Y_2Y_3$.
Therefore, $Y_2^2$ is non-trivial at~$-1$.
Thus, $Y_4$ is not super-strongly non-trivial at~$-1$.
\end{proof}

\subsection{Primality of the ideals generated by functional and trivial relations}

\begin{lemma} \cite[p.~61]{Bom81} \label{functional-relations-ideal-is-prime}
The ideal of $\Qbar[X][Y_1,\dotsc,Y_\numfuns]$ generated by the functional relations is prime.
\end{lemma}

\begin{proof}
Let $I$ denote the kernel of the homomorphism $\Qbar[X][Y_1,\dotsc,Y_\numfuns] \to \powerseries{\Qbar}{X}$ defined by $Y_i \mapsto y_i$.
Since $\powerseries{\Qbar}{X}$ is an integral domain, $I$ is a prime ideal.
The lemma follows by \cite[00JT, Lemma~10.57.7]{stacks}.
\end{proof}

More importantly for applications, the ideal generated by the trivial relations is often prime.
Indeed, the sufficient condition of the following lemma is often satisfied in applications
(for example, the lemma could replace the considerable effort employed for a special case in \cite[5.E]{DO:LGO-PEL}).
This is useful for two reasons:
\begin{enumerate}
\item If the ideal generated by trivial relations is prime, then non-trivial relations are automatically strongly non-trivial and super-strongly non-trivial.
\item In applications, one often constructs ``$v$-adic'' relations for different places and takes their product to get a ``global'' relation.
If the ideal generated by trivial relations is prime, then showing that the factors are non-trivial suffices to show that their product is also non-trivial.
\end{enumerate}

\begin{lemma}
Let $\xi \in \Qbar$.
If the functional relations are generated by elements of $\Qbar[Y_1,\dotsc,Y_\numfuns]$ (that is, by functional relations which do not depend on~$X$), then the ideal of $\Qbar[Y_1,\dotsc,Y_\numfuns]$ generated by the trivial relations at~$\xi$ is prime.
\end{lemma}

\begin{proof}
Let $P_1,\dotsc,P_r$ be elements of $\Qbar[Y_1,\dotsc,Y_\numfuns]$ which generate the functional relations.
Let $I$ and~$\tilde I$ denote the ideals of $\Qbar[Y_1,\dotsc,Y_\numfuns]$ and~$\Qbar[X][Y_1,\dotsc,Y_\numfuns]$, respectively, generated by $P_1,\dotsc,P_r$.
By \cref{functional-relations-ideal-is-prime}, $\tilde I$ is prime.
By an elementary calculation about polynomial rings, $I = \tilde I \cap \Qbar[Y_1,\dotsc,Y_\numfuns]$, so it follows that $I$ is prime.
But $I$ is also equal to the ideal generated by the trivial relations at~$\xi$.
\end{proof}

\subsection{\texorpdfstring{$v$}{v}-adic and global relations}

Let $\relation \in \Qbar[Y_1,\dotsc,Y_n]$ be a homogeneous polynomial.
Let $K$ be a number field which contains $\xi$ and the coefficients of $y_1,\dotsc,y_\numfuns$ and $\relation$.

The following definitions are from \cite[Def.~1.10]{Urb25} and \cite[VII, 5.1]{And89}.

\begin{definition} \leavevmode
\begin{enumerate}
\item A place $v$ of~$K$ is \defterm{relevant for~$\xi$} if $\abs{\xi}_v < \min \{ 1, R_v(y_1), \dotsc, R_v(y_\numfuns) \}$.
\item $\relation$ is a \defterm{$v$-adic relation between $y_1, \dotsc, y_\numfuns$ at~$\xi$}, where $v$ is a relevant place of~$K$ for~$\xi$, if
\[ \relation(y_1^{\van}(\xi), \dotsc, y_\numfuns^\van(\xi)) = 0. \]
\item $\relation$ is a \defterm{global relation between $y_1, \dotsc, y_\numfuns$ at~$\xi$} if it is a $v$-adic relation between $y_1,\dotsc,y_\mu$ at~$\xi$ for every relevant place~$v$ of~$K$.
\end{enumerate}
\end{definition}
The notion of global relation is independent of the choice of~$K$.

\subsection{\texorpdfstring{$(r_v)$}{(r\_v)}-global relations}

We introduce a new definition of ``$(r_v)$-global relations'', mildly generalising the global relations.
This provides a new way to deal with a wrinkle in all applications of G-functions to unlikely intersections, namely that the radii within which unlikely intersections give rise to relations between G-functions might be smaller than the radii of convergence. Previous papers have adopted different approaches to resolve this issue: see \cite[X, 3.1]{And89}, \cite[Lemma~4.2]{Urb25}, and the function~$H$ in \cite[sec.~5.8]{DO:Y1n}. Using $(r_v)$-global relations simplifies the application by removing this extra step, and leads to sharper explicit bounds than the previous approaches.

\begin{definition}
A \defterm{system of radii over~$K$} is a sequence~$(r_v)$, indexed by the places $v$ of~$K$, with values $r_v \in \RR_{>0} \cup \{\infty\}$ for all~$v$.
\end{definition}

\begin{definition}
A system of radii~$(r_v)$ is \defterm{acceptable for $y_1, \dotsc, y_\numfuns$} if:
\begin{enumerate}[(i)]
\item $r_v \leq R_v(y_1,\dotsc,y_\numfuns)$ for all~$v$; and
\item $\sum_v' \log^+(1/r_v) < \infty$.
\end{enumerate}
\end{definition}

If $y_1,\dotsc,y_\numfuns$ are G-functions, then the global radius $\rho(y_1,\dotsc,y_\numfuns)$ is finite by \cite[VI, 5, Main~Thm.]{And89}.
It follows that the radii of convergence $R_v(y_1,\dotsc,y_\numfuns)$ themselves form an acceptable system of radii.


If $(r_v)$ is an acceptable system of radii over~$K$, $\hat K$ is a finite extension of~$K$, and $\hat v$ is a place of~$\hat K$, we shall write $r_{\hat v}$ as an abbreviation for~$r_{(\hat v|_K)}$.
Using this notation, $(r_{\hat v})$ forms an acceptable system of radii for $y_1,\dotsc,y_\numfuns$ over~$\hat K$, which we call the \defterm{extension of $(r_v)$ to~$\hat K$}.

\begin{definition}
Let $(r_v)$ be an acceptable system of radii for $y_1,\dotsc,y_\numfuns$ over~$K$.
Let $\hat K$ be a finite extension of~$K$ which contains $\xi$ and the coefficients of~$\relation$.
\begin{enumerate}
\item A place $\hat v$ of~$\hat K$ is \defterm{$(r_v)$-relevant for~$\xi$} if $\abs{\xi}_{\hat v} < \min \{ 1, r_{\hat v} \}$.
\item $\relation$ is an \defterm{$(r_v)$-global relation between $y_1, \dotsc, y_\numfuns$ at~$\xi$} if it is a $\hat v$-adic relation between $y_1,\dotsc,y_\numfuns$ at~$\xi$ for every $(r_v)$-relevant place~$\hat v$ of~$\hat K$.
\end{enumerate}
\end{definition}

The notion of $(r_v)$-global relation is independent of the choice of~$\hat K$.

In the case $r_v=R_v(y_1,\dotsc,y_\numfuns)$ for all~$v$, $(r_v)$-relevant places and $(r_v)$-global relations are the same as relevant places and global relations, respectively.
For a general acceptable system of radii~$(r_v)$, every $(r_v)$-relevant place is relevant and every global relation is $(r_v)$-global, but the converses need not be true.

\section{Explicit height bounds for general non-trivial global relations} \label{sec:bounds-nt}

In this section, we state and prove explicit height bounds for points where there are non-trivial $(r_v)$-global relations between G-functions.
(Bounds involving super-strongly non-trivial relations are proved in the next section.)
The key ingredient is the explicit inequality \cite[VII, 3.5, Prop.]{And89}.
We correct some minor errors in this inequality, then deduce the explicit versions of \cref{andre-thm-E}.

\subsection{Set-up: Input to theorems} \label{ssec:input-data}


Throughout sections \ref{sec:bounds-nt} and~\ref{sec:bounds-super-strongly-nt}:
\begin{enumerate}[(i)]
\item $y_1, \dotsc, y_\numfuns$ denote G-functions whose $\Qbar(X)$-linear span is closed under $d/dX$; 
\item $\Lambda$ denotes a differential operator over~$\Qbar(X)$ such that $(y_1,\dotsc,y_\numfuns)$ is a solution of~$\Lambda$;
\item $K$ denotes a number field which contains the coefficients of $y_1,\dotsc,y_\numfuns$ and~$\Lambda$;
\item $(r_v)$ denotes an acceptable system of radii for $y_1,\dotsc,y_\numfuns$ over~$K$.
\end{enumerate}

\begin{remark}
The condition that the $\Qbar(X)$-linear span of $y_1,\dotsc,y_\numfuns$ must be closed under $d/dX$ does not restrict the situations in which these height bounds can be applied, because given an arbitrary finite set of G-functions, one can always achieve this condition by appending finitely many derivatives of the original G-functions. Note that $\mu$ always refers to a number of G-functions which satisfy this condition (this condition was also required in \cite[Thm.~2]{Bom81} and \cite[VII, 5.2, Thm.]{And89}).
\end{remark}

The following quantities appear as inputs to the explicit bounds:
\begin{enumerate}[(i)]
\item $\numfuns$, the number of G-functions $y_1,\dotsc,y_\numfuns$;
\item $s$, an integer such that $\max \{ 1, \#\Sin(\Lambda) \} \leq s$;
\item $H$, an upper bound for $H(\Sin(\Lambda))$;
\item $\rho$, denote an upper bound for $\Sigma'_v \log^+(1/r_v)$;
\item $\sigma_y$, an upper bound for $\sigma(y_1,\dotsc,y_\numfuns)$;
\item $\sigma_\Lambda$, an upper bound for $\sigma(\Lambda)$.
\end{enumerate}

In order to simplify the statements of bounds by grouping together terms which are usually negligible, we write
\[ \theta_\Lambda = \max\{ \log(2), \sigma_\Lambda, H \}. \]

\subsection{Statements of explicit height bounds for non-trivial global relations} \label{ssec:ab-explicit-start}

This paper's main theorem is the following refined explicit version of \cref{andre-thm-E}.
All the other theorems of the paper are deduced from this one.

\begin{theorem} \label{andre-bombieri}
Let $\nummonomials,\kappa,\delta \in \ZZ_{\geq 1}$, with $\kappa < \nummonomials$.
Let $m_1,\dotsc,m_\nummonomials$ be monomials of degree~$\delta$ in indeterminates~$Y_1,\dotsc,Y_\numfuns$, such that the $\Qbar(X)$-linear span of $\{ m_i(y_1,\dotsc,y_\numfuns) : 1 \leq i \leq \nummonomials \}$ is closed under~$d/dX$.

Let $\xi \in \Qbar \setminus \{0\}$ be an ordinary point or good apparent singularity of~$\Lambda$.
Suppose that there exist $(r_v)$-global relations $\relation_1,\dotsc,\relation_\kappa$ between $y_1,\dotsc,y_\numfuns$ at~$\xi$, such that:
\begin{enumerate}
\item $\relation_1,\dotsc,\relation_\kappa$ lie in the $\Qbar$-span of $m_1,\dotsc,m_\nummonomials$;
\item $\relation_1,\dotsc,\relation_\kappa$ are $\Qbar$-linearly independent modulo the trivial relations at~$\xi$.
\end{enumerate}
Then:
\begin{align}
    h(\xi) < {}
  & 8s\frac{\nummonomials^2}{\kappa^2}(\nummonomials-\tfrac{1}{2}) \bigl( (\log(\delta)+1)\sigma_y + \rho \bigr)
\notag
\\& + 6\frac{\nummonomials}{\kappa} \log(2) + \frac{1}{2s}(\log(\delta)+1)\sigma_\Lambda + \frac{1}{s}H.
\label{eqn:hxi-general-bound}
\end{align}
\end{theorem}

\begin{remark} \label{andre-bombieri-refinements}
Compared with \cref{andre-thm-E}, and its ``direct'' explicit version \cref{andre-5.2-explicit}, \cref{andre-bombieri} takes into account the following additional information about the relation(s) between $y_1,\dotsc,y_\numfuns$ at~$\xi$, namely:
\begin{enumerate}
\item the dimension of a suitable space of polynomials containing~$\relation$ (to deduce \cref{andre-5.2-explicit}, one uses the space of all homogeneous polynomials of the same degree as~$\relation$; in applications, $\relation$ may lie in a smaller subspace);
\item multiple linearly independent relations of the same degree satisfied by the evaluations at~$\xi$ of $y_1,\dotsc,y_\numfuns$ (\cref{andre-5.2-explicit} simply uses a single relation).
\end{enumerate}
\end{remark}

Usually, $s\nummonomials/\kappa$ is reasonably large and \eqref{eqn:hxi-general-bound} is dominated by the term
\[ 8s\frac{\nummonomials^3}{\kappa^2} \bigl( (\log(\delta)+1)\sigma_y + \rho \bigr). \]
However, $\sigma_y$ and~$\rho$ could be~$0$ or small compared with $\sigma_\Lambda$ and~$H$, so the second line of~\eqref{eqn:hxi-general-bound} is not always beaten by the ``$-\tfrac{1}{2}$'' term on the first line.
(The upper bound for $\sigma(\Lambda)$ at \cite[VI, 4, Thm.]{And89} is not sufficient to remove the $\sigma_\Lambda$ term from~\eqref{eqn:hxi-general-bound}.)
We can simplify the bound in \cref{andre-bombieri} to the following:

\begin{corollary} \label{andre-bombieri-simp}
In the situation of \cref{andre-bombieri}, we have:
\begin{equation} \label{eqn:hxi-general-bound-simp}
h(\xi) < 8s\frac{\nummonomials^3}{\kappa^2} \max \bigl\{ (\log(\delta)+1) \sigma_y + \rho, \; \frac{\kappa}{s\nummonomials} (\log(\delta)+2) \theta_\Lambda \bigr\}.
\end{equation}
\end{corollary}

\begin{proof}
Since $\nummonomials/\kappa \geq 1$ and $s \geq 1$, the lower line of~\eqref{eqn:hxi-general-bound} is at most
\[ 8 \cdot \tfrac{1}{2} \cdot \frac{\nummonomials}{\kappa}(\log(\delta)+2)\theta_\Lambda. \qedhere \]
\end{proof}

In~\eqref{eqn:hxi-general-bound-simp}, the first element inside the $\max$ is usually much larger than the second (while $\sigma_y$ and~$\rho$ are usually similar in size).
Subsequent theorems in this paper are written with a bound similar in form to~\eqref{eqn:hxi-general-bound-simp}, although in each theorem one could write down a more precise bound similar to~\eqref{eqn:hxi-general-bound}.

The following corollary is the explicit version of \cref{andre-thm-E}, that is, of the bound for $\Sha_\delta$ in \cite[VII, 5.2, Thm.]{And89}.
The corollary is deduced from \cref{andre-bombieri} by restricting to $\kappa=1$ and by taking $\{m_1,\dotsc,m_\nummonomials\}$ to be the set of all monomials of degree~$\delta$ in $Y_1, \dotsc, Y_\numfuns$.
The number of such monomials is
\[ \nummonomials = \binom{\numfuns+\delta-1}{\numfuns-1} \leq \numfuns\delta^{\numfuns-1}. \]
This is greater than $\kappa=1$ (so that \cref{andre-bombieri} applies) whenever $\numfuns \geq 2$.

\begin{corollary} \label{andre-5.2-explicit}
Let $\delta \in \ZZ_{\geq 1}$.
Suppose that $\numfuns \geq 2$.
For all $\xi \in \Qbar \setminus\{0\}$, if $\xi$ is an ordinary point or a good apparent singularity of~$\Lambda$ and there exists a non-trivial $(r_v)$-global relation of degree~$\delta$ between $y_1,\dotsc,y_\numfuns$ at~$\xi$, then
\begin{align*}
h(\xi) < 8s\numfuns^3 \delta^{3(\numfuns-1)} \max \bigl\{ (\log(\delta)+1) \sigma_y + \rho, \; \frac{1}{s\numfuns\delta^{\numfuns-1}} (\log(\delta)+2) \theta_\Lambda \bigr\}.
\end{align*}
\end{corollary}

\subsection{Correction to explicit bound in \texorpdfstring{\cite[VII, 3.5]{And89}}{[And89, VII, 3.5]}}

The primary ingredient for proving \cref{andre-bombieri} is the explicit inequality \cite[VII, 3.5, Prop.]{And89}.
This paper does not discuss the full proof of this inequality (found in \cite[IV, 5 and VII, 2 and~3]{And89}), but we fix some errors in it.
These errors do not affect \cite[VII, 5.2, Thm.]{And89} because it does not give explicit values for constants, but they are significant for obtaining an explicit height bound.
The corrected statement of \cite[VII, 3.5, Prop.]{And89} is as follows.

\begin{proposition} \label{andre-explicit-bound-corrected}
Let $\tau$ be a positive real number.
Let $\kappa$ be an integer satisfying $0 < \kappa < \numlinmons$.
Let $\nu = \numlinmons-\kappa$.
Suppose that $y_1,\dotsc,y_\numlinmons$ are $\Qbar(X)$-linearly independent.

Let $\xi \in K \setminus \{0\}$ be an ordinary point or a good apparent singularity of~$\Lambda$.
Let $V$ be a finite set of places of~$K$ such that
\begin{equation} \label{eqn:V-condition}
\abs{\xi}_v < \min\{ 1, R_v(y_1), \dotsc, R_v(y_\numlinmons) \} \cdot \min\{ 1, \abs{2}_v^{-1} \}
\end{equation}
for all $v \in V$.
Suppose that there exist $\kappa$ $K$-linearly independent homogeneous linear polynomials $\relation_1, \dotsc, \relation_\kappa \in K[Y_1,\dotsc,Y_\numlinmons]$ such that, for every $v \in V$ and every $j=1,\dotsc,\kappa$, $\relation_j$ is a non-trivial $v$-adic relation between $y_1,\dotsc,y_\numlinmons$ at~$\xi$.
Then
\begin{equation} \label{eqn:xi_v-inequality-3.5}
{\sum_{v \in V}}' \log \abs{\xi}_v + \Bigl( \frac{\nu}{\numlinmons}(1+\tau) + \tau(2s-1) \Bigr) h(\xi) \geq -\refC{andres-c4}(\Lambda, \tau, \underline y),
\end{equation}
\createC{andres-c4}
where
\begin{align*}
    \refC{andres-c4}(\Lambda, \tau, \underline y)
  & = \bigl( s\tau + \frac{\nu}{\numlinmons}(1+\tau) + 1 \bigr) \log(2)
      + \tau \sigma(\Lambda) + 2 \tau H(\Sin(\Lambda))
\\& \phantom{{}={}} + \nu( 1+\tfrac{1}{\tau} ) \sigma(y_1,\dotsc,y_\numlinmons)
      + \bigl( 1 + \nu( 1+\tfrac{1}{\tau} ) \bigr) \rho(y_1,\dotsc,y_\numlinmons).
\end{align*}
\end{proposition}



\Cref{andre-explicit-bound-corrected} incorporates two corrections of \cite[VII, 3.5, Prop.]{And89}.

The first correction is that this paper defines $s = \max\{1, \#\Sin(\Lambda)\}$ instead of $s=\#\Sin(\Lambda)$.
Indeed, \cite[IV, 5.4, Prop.]{And89}, used in the proof of \cite[VII, 3.5, Prop.]{And89}, is incorrect when $\#\Sin(\Lambda)=0$, as illustrated by constant G-functions ($\Lambda=0$).
The error in the proof of \cite[IV, 5.4, Prop.]{And89} is that it assumes that
\[ m \leq n \Rightarrow (\# \Sin(\Lambda)-1)m+c \leq (\#\Sin(\Lambda)-1)n+c \]
in order to obtain the bound for $h_{v,n}(y)$.
This implication is false when $\#\Sin(\Lambda) = 0$, but it may be fixed by replacing $\#\Sin(\Lambda)$ by~$1$ in this case.

The second correction is that, if $\xi$ is an apparent singularity, \cref{andre-explicit-bound-corrected} requires it to be good.
To see that \cite[VII, 3.5, Prop.]{And89}\ is not always true when $\xi$ is a bad apparent singularity, consider the example
\begin{align*}
    y_1 & = \log(1+X),
\\  y_2 & = X-a,
\end{align*}
where $a \in \ZZ \setminus \{-1\}$.
The vector $(y_1,y_2)^t$ is a solution of the differential operator
\[ \Lambda = \frac{d}{dX} - \Gamma
   = \frac{d}{dX} - \begin{pmatrix} 0 & 1/(X-a)(X+1) \\ 0 & 1/(X-a) \end{pmatrix}. \]
The point~$a$ is a bad apparent singularity of~$\Lambda$.
The polynomial~$Y_2$ is a global relation between $y_1, y_2$ at~$a$.
Taking $V$ to be the set of places of~$\QQ$ given by prime factors of~$a$, if \cite[VII, 3.5, Prop.]{And89} applied to $\xi=a$, then we would obtain
\begin{equation} \label{eqn:ha-bound}
-h(a) + \bigl(\tfrac{1}{2} + \tfrac{3}{2}\tau \bigr) h(a) \geq -\refC{andres-c4}(\Lambda, \tau, (y_1,y_2)).
\end{equation}
The constant $\refC{andres-c4}(\Lambda,\tau,(y_1,y_2))$ is independent of~$a$.
Hence, if we take $\tau < \tfrac{1}{3}$, then \eqref{eqn:ha-bound} contradicts the fact that $a$ may be arbitrarily large.
The error in the proof in \cite[VII, 3.3]{And89} is the assertion that solutions of $\Lambda$ are `uniquely determined by their ``initial'' value at~$\xi$'.
This is precisely the condition that $\xi$ should be an ordinary point or a \emph{good} apparent singularity.

\subsection{Proof of Theorem~\ref{andre-bombieri}} \label{ssec:ab-explicit-end}

In order to prove \cref{andre-bombieri}, we first prove the special case of linear relations, under the additional assumption that the G-functions $y_1,\dotsc,y_\numfuns$ are linearly independent (a strengthening of \cite[Thms. 2 and~3]{Bom81}).
We will then prove a lemma allowing us to remove the linear independence assumption, and finally generalise to relations of degree greater than~$1$.

\begin{lemma} \label{andre-bombieri-linear-indep}
In the setting of \cref{andre-bombieri}, suppose that $\delta=1$.
Let $\nummonomials=\numfuns$ and $m_i=Y_i$ for each $i=1,\dotsc,\numfuns$.
Suppose that $y_1, \dotsc,y_\numfuns$ are $\Qbar(X)$-linearly independent.
Let $\xi \in \Qbar \setminus \{0\}$ be an ordinary point or good apparent singularity of~$\Lambda$.
Suppose that there exist $\kappa$ $\Qbar$-linearly independent $(r_v)$-global relations of degree~$1$ between $y_1,\dotsc,y_\numfuns$ at~$\xi$.
Then
\begin{align*}
    h(\xi) < {}
  & \frac{s\nummonomials^2}{\kappa^2} (8\nummonomials - 4) \bigl( \sigma_y + \rho \bigr)
    + 6\frac{\nummonomials}{\kappa} \log(2) + \frac{1}{2s}\sigma_\Lambda + \frac{1}{s}H.
\end{align*}
\end{lemma}

\begin{proof}
After replacing~$K$ by a finite extension, and $(r_v)$ by its corresponding extension, we may assume that $K$ contains $\xi$ and the coefficients of the $(r_v)$-global relations.
Let $V$ denote the set of places~$v$ of~$K$ satisfying
\begin{align*}
    \abs{\xi}_v   & < \min\{ 1, r_v \} \text{ if } v \text{ is non-archimedean,}
\\  \abs{2\xi}_v  & < \min\{ 1, r_v \} \text{ if } v \text{ is archimedean.}
\end{align*}
As in the proof of \cite[VII, 5.2, Thm.]{And89}, a simple calculation shows that
\begin{align}
    {\sum_{v \in V}}' \log \abs{\xi}_v
    = -{\sum_{v \in V}}' \log^+ (1/\abs{\xi}_v)
  & = -h(\xi) + {\sum_{v \not\in V}}' \log^+ (1/\abs{\xi}_v)
\notag
\\&
    \leq -h(\xi) + \rho + \log(2).
\label{eqn:xi_v-inequality-5.2}
\end{align}
Combining \cref{andre-explicit-bound-corrected} with~\eqref{eqn:xi_v-inequality-5.2} yields
\begin{equation*}
\tfrac{1}{\numlinmons} \bigl( -\kappa + \tau(2s\numlinmons-\kappa) \bigr) h(\xi) + \rho + \log(2) \geq -\refC{andres-c4}(\Lambda, \tau, \underline y).
\end{equation*}
Expanding the value of $\refC{andres-c4}(\Lambda, \tau, \underline y)$, we obtain
\begin{align}
    \tfrac{1}{\numlinmons} \bigl( \kappa + \tau(\kappa-2s\numlinmons) \bigr) h(\xi)
  & \leq \bigl( (s + 1 - \tfrac{\kappa}{\numlinmons})\tau + 3 - \tfrac{\kappa}{\numlinmons} \bigr) \log(2)
         + \tau\sigma_\Lambda + 2\tau H
\notag
\\& \phantom{{}={}} + (\numlinmons-\kappa)(1+\tfrac{1}{\tau}) \sigma_y + \bigl(2 + (\numlinmons-\kappa)(1+\tfrac{1}{\tau}) \bigr) \rho.
\label{eqn:hxi-tau-inequality}
\end{align}
As noted in \cite[VII, 3.5, Rmk.~2]{And89}, this inequality is non-trivial (the coefficient on the left is positive) when $\tau < \sfrac{\kappa}{(2s\numlinmons-\kappa)}$.
(The claim in the proof of \cite[VII, 5.2, Thm.]{And89}\ that, for $\kappa=1$, the inequality is non-trivial when $\tau < \sfrac{(\numlinmons-1)}{(2s\numlinmons-1)}$ appears to be a typo.)

We choose $\tau=\kappa/4s\numlinmons$, which is close to the optimal value.
We obtain:
\begin{align*}
    \frac{1}{\numlinmons} \bigl( \kappa + \tau(\kappa-2s\numlinmons) \bigr)
  & = \frac{1}{\numlinmons} \Bigl( \frac{\kappa}{2} + \frac{\kappa^2}{4s\numlinmons} \Bigr) > \frac{\kappa}{2\numlinmons},
\\  (s+1-\tfrac{\kappa}{\numlinmons})\tau + 3-\tfrac{\kappa}{\numlinmons}
  & < \frac{\kappa}{4\numlinmons} + \frac{\kappa}{4s\numlinmons} + 3 - \frac{\kappa}{\numlinmons} 
    < 3,
\\  (\numlinmons-\kappa)( 1+\tfrac{1}{\tau} ) < 2+(\numlinmons-\kappa)( 1+\tfrac{1}{\tau} )
  & = 2+(\numlinmons-\kappa) \Bigl( \frac{4s\numlinmons}{\kappa} + 1 \Bigr)
    \leq \frac{s\numlinmons}{\kappa}(4\numlinmons-2).
\end{align*}
Substituting these into~\eqref{eqn:hxi-tau-inequality} proves the lemma.
\end{proof}

In order to remove the condition that $y_1,\dotsc,y_\numfuns$ are $\Qbar(X)$-linearly independent from \cref{andre-bombieri-linear-indep}, we need to check that, when we discard some G-functions which are linearly dependent on the others, there is still a $(r_v)$-global relation among the remaining G-functions.
Note that \cref{linear-relations-redundant-variables} is not true for global relations rather than $(r_v)$-global relations, because we might have $R_v(y_1,\dotsc,y_{\nummonomials'}) > R_v(y_1,\dotsc,y_\nummonomials)$, in which case a global relation between $y_1,\dotsc,y_\nummonomials$ involving only $Y_1,\dotsc,Y_{\nummonomials'}$ might not be a global relation between $y_1,\dotsc,y_{\nummonomials'}$.

\begin{lemma} \label{linear-relations-redundant-variables}
Let $y_1,\dotsc,y_\nummonomials \in\powerseries{\Qbar}{X}$.
Let $\xi \in \Qbar \setminus \{0\}$.
If there are $\kappa$ $(r_v)$-global relations of degree~$1$ between $y_1,\dotsc,y_\nummonomials$ at~$\xi$, $\Qbar$-linearly independent modulo the trivial relations at~$\xi$, then there exists a subset $J \subset \{ 1,\dotsc,\nummonomials \}$ such that:
\begin{enumerate}
\item the $y_j$ for $j \in J$ form a $\Qbar(X)$-basis for the $\Qbar(X)$-span of $y_1,\dotsc,y_\nummonomials$;
\item there are $\kappa$ $\Qbar$-linearly independent $(r_v)$-global relations of degree~$1$ between $\{ y_j : j \in J \}$ at~$\xi$.
\end{enumerate}
\end{lemma}

\begin{proof}
Let $R$ denote the discrete valuation ring $\Qbar[X]_{(X-\xi)}$ and let $M$ denote the $R$-module spanned by $y_1,\dotsc,y_\nummonomials$.
Since $M$ is a torsion-free module over a DVR, it is free.
Choose $J \subset \{ 1, \dotsc, \nummonomials \}$ such that the images of $\{ y_j : j \in J \}$ form a $\Qbar$-basis for $M/(X-\xi)M$.
By Nakayama's lemma, $\{ y_j : j \in J \}$ is an $R$-basis for~$M$, and hence, a $\Qbar(X)$-basis for $M \otimes_R \Qbar(X)$.
Thus $J$ satisfies~(1).

Choose $t_{ij} \in \Qbar[X]$ (for $i=1, \dotsc, \nummonomials$ and $j \in J$) and $u \in \Qbar[X]$ such that $u(\xi) \neq 0$ and, for each~$i$,
\begin{equation} \label{eqn:linear-functional-relations}
uy_i = \sum_{j \in J} t_{ij}y_j.
\end{equation}

Let $\relation_1,\dotsc,\relation_\kappa$ be the $(r_v)$-global relations of degree~$1$ between $y_1,\dotsc,y_\nummonomials$ at~$\xi$ given by the hypothesis of the lemma.
For each $k=1,\dotsc,\kappa$, let 
\begin{align*}
    \relation'_k(Y_j : j \in J) & = \relation_k \biggl( \sum_{j \in J} t_{1j}(\xi) Y_j, \, \dotsc, \sum_{j \in J} t_{\nummonomials j}(\xi) Y_j \biggr).
\end{align*}
Thanks to~\eqref{eqn:linear-functional-relations}, each $\relation'_k$ is an $(r_v$)-global relation of degree~$1$ between $y_j$ (for $j \in J$) at~$\xi$.
Therefore, to prove~(2), it suffices to show that $\relation'_1, \dotsc, \relation'_\kappa$ are $\Qbar$-linearly independent.

Indeed, assume that there is a linear dependence
\begin{equation} \label{eqn:Q'-dependence}
b_1\relation'_1 + \dotsb + b_\kappa \relation'_\kappa = 0,
\end{equation}
where $b_1, \dotsc, b_\kappa \in \Qbar$ are not all zero.
Let
\[ \tilde\relation(X)(Y_1, \dotsc, Y_\nummonomials) = \sum_{k=1}^\kappa b_k \relation_k \Bigl( uY_1 - \sum_{j \in J} t_{1j} Y_j, \, \dotsc, \, uY_\nummonomials - \sum_{j \in J} t_{\nummonomials j} Y_j \Bigr). \]
Thanks to~\eqref{eqn:linear-functional-relations}, 
\[ \tilde\relation(X)(y_1,\dotsc,y_\nummonomials) = \sum_{k=1}^\kappa b_k \relation_k(0,\dotsc,0) = 0. \]
Thus, $\tilde\relation$ is a functional relation between $y_1,\dotsc,y_\nummonomials$, so $\tilde \relation(\xi)(Y_1,\dotsc,Y_\nummonomials)$ is a trivial relation at~$\xi$.
However, by the definition of~$\relation'_k$ and by~\eqref{eqn:Q'-dependence},
\[ \tilde \relation(\xi)(Y_1,\dotsc,Y_\nummonomials) = \sum_{k=1}^\kappa b_k \bigl( u(\xi) \relation_k - \relation'_k \bigr) = \sum_{k=1}^\kappa b_k u(\xi) \relation_k. \]
Since $u(\xi) \neq 0$, this contradicts the hypothesis that $\relation_1,\dotsc,\relation_\kappa$ are $\Qbar$-linearly independent modulo the trivial relations at~$\xi$.
\end{proof}

Now we are ready to prove \cref{andre-bombieri}.

\begin{proof}[Proof of \cref{andre-bombieri}]
Write $z_i = m_i(y_1,\dotsc,y_\numfuns) \in \powerseries{\Qbar}{X}$ for $1 \leq i \leq \nummonomials$.
Apply \cref{linear-relations-redundant-variables} to $z_1,\dotsc,z_\nummonomials$ and~$\xi$ to obtain $J \subset \{ 1,\dotsc,\nummonomials \}$.

By the hypothesis of the theorem, the $\Qbar(X)$-linear span of $z_j$ (for $j \in J$) is a $\partial$-submodule of $\powerseries{\Qbar}{X}$.
Hence, there exists a differential operator~$\Lambda_\cM$ with solution $(z_j : j \in J)$.
Let $\Lambda^{\odot\delta}$ denote the $\delta$-th symmetric power of the differential operator~$\Lambda$ whose solution is $(y_1,\dotsc,y_\numfuns)$.
Thus, $\Lambda_\cM$ is a sub-operator of~$\Lambda^{\odot\delta}$.

According to \cite[VII, 4.1]{And89}, we have
\begin{align*}
    \rho(z_j : j \in J) & \leq \rho(\underline y^{\odot \delta}) = \rho(\underline y),
\\  \sigma(z_j : j \in J) & \leq \sigma(\underline y^{\odot \delta}) \leq \sigma(\underline y)(\log(\delta)+1),
\\  \sigma(\Lambda_\cM) & \leq \sigma(\Lambda^{\odot \delta}) \leq \sigma(\Lambda)(\log(\delta)+1),
\\  \Sin(\Lambda_\cM) & \subset \Sin(\Lambda^{\odot\delta}) = \Sin(\Lambda).
\end{align*}

Thus the theorem follows from \cref{andre-bombieri-linear-indep} applied to $\{ z_j : j \in J \}$.
\end{proof}

\section{Explicit height bounds for super-strongly non-trivial global relations} \label{sec:bounds-super-strongly-nt}

In this section, we state and prove explicit height bounds when there are \emph{super-strongly} non-trivial $(r_v)$-global relations between G-functions.
The idea is that, given a global relation~$\relation$ of degree~$\delta$, we can obtain many global relations of degree~$\delta+\gamma$ by multiplying $\relation$ by different polynomials of degree~$\gamma$.
For suitably chosen~$\gamma$, we get a better bound by applying \cref{andre-bombieri} to these relations of degree~$\delta+\gamma$ instead of the original~$P$.
The ``super-strongly non-trivial'' condition ensures that the multiples of~$\relation$ are linearly independent modulo trivial relations.

Hilbert functions are used to count the dimensions of the spaces of polynomials or of multiples of~$\relation$ of degree~$\delta+\gamma$, modulo trivial relations.
We recall the definition.

\begin{definition}
Let $K$ be a field.
Let $I$ be a homogeneous ideal in $K[Y_1,\dotsc,Y_\numfuns]$.
We write $HF_I \colon \ZZ_{\geq 0} \to \ZZ_{\geq 0}$ to denote the \defterm{Hilbert function} of the graded ring $K[Y_1,\dotsc,Y_\numfuns]/I$, that is,
\[ HF_I(\gamma) = \dim_K ( K[Y_1,\dotsc,Y_\numfuns]_\gamma/I_\gamma ), \]
where the subscript~$\gamma$ means the subspace of homogeneous polynomials of degree~$\gamma$.
\end{definition}

 \subsection{Statements for super-strongly non-trivial global relations} \label{ssec:ab-super-strong-start}

We state several versions of explicit height bounds for super-strongly non-trivial relations.
\Cref{andre-bombieri-HF} requires knowledge of specific values of the Hilbert function.
It involves a parameter~$\gamma$, which may be chosen to optimise the bound.
\Cref{andre-bombieri-HF-bounds} is stated in terms of bounds for the Hilbert function, and makes a good choice of~$\gamma$ for bounds of the given form.
The last two versions remove Hilbert functions from the statements: \cref{andre-bombieri-strong-general} applies general bounds on the Hilbert function, while \cref{andre-bombieri-homog-alg-ind} improves the bound in the special case of no functional relations.

\begin{theorem} \label{andre-bombieri-HF}
Let $\delta \in \ZZ_{\geq 1}$ and $\gamma \in \ZZ_{\geq 0}$.
Let $I \subset \Qbar(X)[Y_1,\dotsc,Y_\numfuns]$ be the homogeneous ideal generated by the functional relations between $y_1,\dotsc,y_\numfuns$.

Let $\xi \in \Qbar \setminus \{0\}$ be an ordinary point or good apparent singularity of~$\Lambda$.
Suppose that there exists a \textbf{super-strongly} non-trivial $(r_v)$-global relation of degree~$\delta$ between $y_1,\dotsc,y_\numfuns$ at~$\xi$.
Then the bound~\eqref{eqn:hxi-general-bound} holds with $\delta$ replaced by $\delta+\gamma$ and with $\nummonomials = HF_I(\delta+\gamma)$ and $\kappa = HF_I(\gamma)$.
\end{theorem}

The most convenient form of bounds for the Hilbert function is
\begin{equation} \label{eqn:HF-general-estimate}
A\binom{\gamma+\eta}{\eta} \leq HF_I(\gamma) \leq B\binom{\gamma+\eta}{\eta},
\end{equation}
where $\eta$ is the dimension of $V$, the closed subscheme of~$\PP^{\numfuns-1}$ cut out by~$I$.
Bounds of this form are consistent with the well-known asymptotic
\begin{equation} \label{eqn:HF-asymptotic}
HF_I(\gamma) \sim \frac{\deg(V)}{\dim(V)!} \gamma^{\dim(V)} \text{ as } \gamma \to \infty.
\end{equation}

Given a bound for the Hilbert function of the form~\eqref{eqn:HF-general-estimate}, 
we choose $\gamma = 2\delta - \lceil \tfrac{1}{2}\eta \rceil$, as long as this is positive. When $\delta$ is large compared with~$\eta$, this appears to be the optimal value for~$\gamma$.
In particular, if $\eta \leq 4$, then this is the optimal $\gamma$ for all $\delta$ such that it is positive.
Using this value of~$\gamma$ and bounds for binomial coefficients (\cref{binom-delta-gamma-bounds}), as well as the fact that $\delta+\gamma \leq 3\delta$, we deduce the following theorem from \cref{andre-bombieri-HF}.

\begin{theorem} \label{andre-bombieri-HF-bounds}
Let $I \subset \Qbar(X)[Y_1,\dotsc,Y_\numfuns]$ be the homogeneous ideal generated by the functional relations between $y_1,\dotsc,y_\numfuns$.
Let $\eta$ be a positive integer and 
let $\delta$ be an integer satisfying $\delta \geq \tfrac{1}{4}(\eta+2)$.
Let $A$, $B$ be positive real numbers satisfying~\eqref{eqn:HF-general-estimate}
for all $\gamma \geq 2\delta-\lceil \frac{1}{2}\eta \rceil$.

Let $\xi \in \Qbar \setminus \{0\}$ be an ordinary point or good apparent singularity of~$\Lambda$.
Suppose that there exists a super-strongly non-trivial $(r_v)$-global relation of degree~$\delta$ between $y_1,\dotsc,y_\numfuns$ at~$\xi$.
Then
\begin{align}
    h(\xi)
  & < 8s\frac{B^3}{A^2} \refC{hxi-strong-mu-kappa}(\eta)^2 \refC{hxi-strong-mu-mult}(\eta) \delta^\eta
    \max \Bigl\{ (\log(3\delta)+1) \sigma_y + \rho, \; \frac{A}{Bs\refC{hxi-strong-mu-kappa}(\eta)} (\log(3\delta)+2) \theta_\Lambda \Bigr\},
\end{align}
where
$\createC{hxi-strong-mu-kappa} \createC{hxi-strong-mu-mult}$
\begin{align*}
    \refC{hxi-strong-mu-kappa}(\eta)
  & =
    \begin{cases}
        \displaystyle \vphantom{\Bigg(} \frac{2}{(4k+1)(4k)} \binom{5k+2}{4k-1}
      & \text{ if } \eta=4k-1 \text{ with } k \in \ZZ_{\geq1},
    \\  \displaystyle \vphantom{\Bigg(} \frac{1}{4k+2} \binom{5k+3}{4k+1}
      & \text{ if } \eta=4k+1 \text{ with } k \in \ZZ_{\geq0},
    \\  \vphantom{\dfrac{0}{0}} \tfrac{3}{2} \refC{hxi-strong-mu-kappa}(\eta-1)
      & \text{ if $\eta$ is even,}
    \end{cases}
\end{align*}
\begin{align*}
    \refC{hxi-strong-mu-mult}(\eta)
  & = 
    \begin{cases}
        \displaystyle \vphantom{\Bigg(} \frac{3^\eta}{\eta!}
      & \text{ if $\eta$ is even,}
    \\  \displaystyle \vphantom{\Bigg(} \frac{5}{3} \cdot \frac{3^\eta}{\eta!}
      & \text{ if $\eta$ is odd.}
    \end{cases}
\end{align*}
\end{theorem}

\begin{remark}
It is clear that $\refC{hxi-strong-mu-kappa}(\eta) \leq \refC{hxi-smk-mult} 2^{5\eta/4} \leq \refC{hxi-smk-mult} 3^\eta$ for some $\newC{hxi-smk-mult}$, so $\refC{hxi-strong-mu-kappa}(\eta)^2 \refC{hxi-strong-mu-mult}(\eta) \leq \newC* 3^{3\eta} / \eta!$. This goes to zero as $\eta \to \infty$.
\end{remark}

The following corollary requires no information about the functional relations except the dimension and degree of the projective variety which they cut out.
Since $\eta \leq \numfuns-1$, this is an explicit version of the bound for~$\Sha'_\delta$ in \cite[VII, 5.2, Thm.]{And89}\ (restricted to super-strongly non-trivial relations).

\begin{corollary} \label{andre-bombieri-strong-general}
Let $\eta$ and~$D$ denote the dimension and the degree, respectively, of the closed subvariety of~$\PP^{\numfuns-1}_{\Qbar(X)}$ cut out by the functional relations between $y_1,\dotsc,y_\numfuns$.
Suppose that $\eta \geq 1$ (in other words, at least one of the ratios $y_i/y_j$ is transcendental over~$\Qbar(X)$).
Let $\delta$ be an integer satisfying $\delta \geq D + \tfrac{1}{4}\eta$.

Let $\xi \in \Qbar \setminus \{0\}$ be an ordinary point or good apparent singularity of~$\Lambda$.
Suppose that there exists a super-strongly non-trivial $(r_v)$-global relation of degree~$\delta$ between $y_1,\dotsc,y_\numfuns$ at~$\xi$.
Then
\begin{align*}
    h(\xi)
  & < 2^{2\eta+3}sD \refC{hxi-strong-mu-kappa}(\eta)^2 \refC{hxi-strong-mu-mult}(\eta) \delta^\eta
    \max \Bigl\{ (\log(3\delta)+1) \sigma_y + \rho, \; \frac{1}{2^\eta s\refC{hxi-strong-mu-kappa}(\eta)} (\log(3\delta)+2) \theta_\Lambda \Bigr\}.
\end{align*}
\end{corollary}

\begin{proof}
We have $2\delta - \lceil \tfrac{1}{2}\eta \rceil \geq 2D$.
Hence, by \cref{HF-bounds-2}, the Hilbert function of the ideal generated by functional relations satisfies \eqref{eqn:HF-general-estimate} for $\gamma = 2\delta - \lceil \tfrac{1}{2} \eta \rceil$ with
\[ A = 2^{-\eta}D, \quad B=D. \]
Thus $B/A = 2^\eta$ and
$B^3/A^2 = 4^\eta D$.
The corollary follows by \cref{andre-bombieri-HF-bounds}.
\end{proof}

When $y_1,\dotsc,y_\numfuns$ are homogeneously algebraically independent, that is, there are no non-zero functional relations, the Hilbert function is known exactly:
\begin{equation*}
HF_{\{0\}}(\gamma) = \dim(K[Y_1,\dotsc,Y_\numfuns]_\gamma) = \binom{\gamma+\numfuns-1}{\numfuns-1}.
\end{equation*}
In other words, \eqref{eqn:HF-general-estimate} holds with $A=B=1$ and $\eta=\numfuns-1$.
In this case, the ideal of trivial relations at any~$\xi$ is zero, so non-trivial relations are automatically super-strongly non-trivial.
Hence we obtain the following corollary of \cref{andre-bombieri-HF-bounds}.

\begin{corollary} \label{andre-bombieri-homog-alg-ind}
Suppose that $\numfuns \geq 2$.
Let $\delta$ be an integer satisfying $\delta \geq \tfrac{1}{4}(\numfuns+1)$.
Suppose that $y_1,\dotsc,y_\numfuns$ are homogeneously algebraically independent over~$\Qbar(X)$.
Let $\xi \in \Qbar \setminus \{0\}$ be an ordinary point or good apparent singularity of~$\Lambda$.
Suppose that there exists a non-zero $(r_v)$-global relation of degree~$\delta$ between $y_1,\dotsc,y_\numfuns$ at~$\xi$.
Then
\begin{align}
    h(\xi)
  & < 8s \refC{hxi-strong-mu-kappa}(\numfuns-1)^2 \refC{hxi-strong-mu-mult}(\numfuns-1) \delta^{\numfuns-1}
\notag
\\& \qquad\quad \cdot \max \Bigl\{ (\log(3\delta)+1) \sigma_y + \rho, \; \frac{1}{s\refC{hxi-strong-mu-kappa}(\numfuns-1)} (\log(3\delta)+2) \theta_\Lambda \Bigr\}.
\label{eqn:hxi-strong-bound}
\end{align}
\end{corollary}

\subsection{Proof of Theorem~\ref{andre-bombieri-HF}}

The proof of \cref{andre-bombieri-HF} relies on the fact that the Hilbert function of a homogeneous ideal of $K(X)[Y_1,\dotsc,Y_\numfuns]$ is unchanged by specialisation $X \mapsto \xi$, as described by the following lemma and corollary.

\begin{lemma} \label{lift-basis}
Let $I$ be a homogeneous ideal of $K(X)[Y_1,\dotsc,Y_\numfuns]$.
Let $\tilde I = I \cap K[X][Y_1,\dotsc,Y_\numfuns]$.
For $\xi \in K$, let $I_\xi \subset K[Y_1,\dotsc,Y_\numfuns]$ denote the specialisation of $\tilde I$ under $X \mapsto \xi$.
Let $\gamma \in \ZZ_{\geq0}$.

Let $\cB$ be a subset of $K[Y_1,\dotsc,Y_\numfuns]_\gamma$ whose image in $K[Y_1,\dotsc,Y_\numfuns]/I_\xi$ forms a $K$-basis for $K[Y_1,\dotsc,Y_\numfuns]_\gamma/I_{\xi,\gamma}$.
Then the image of~$\cB$ in $K(X)[Y_1,\dotsc,Y_\numfuns]/I$ forms a $K(X)$-basis for $K(X)[Y_1,\dotsc,Y_\numfuns]_\gamma/I_\gamma$.
\end{lemma}

\begin{proof}
Write $R = K(X)[Y_1,\dotsc,Y_\numfuns]$, $\tilde R = K[X][Y_1,\dotsc,Y_\numfuns]$ and $R_\xi = K[Y_1,\dotsc,Y_\numfuns]$.

Since $X-\xi$ is invertible in~$R$ and $\tilde I_\gamma = I_\gamma \cap \tilde R$, we have
\[  (X-\xi)\tilde I_\gamma = \tilde I_\gamma \cap (X-\xi)\tilde R = \ker(\spmap_\xi \colon \tilde I_\gamma \to I_{\xi,\gamma}). \]
Hence the specialisation map $X \mapsto \xi$ induces an isomorphism of $K[X]$-modules
\begin{equation} \label{eqn:sp-isom}
\tilde I_\gamma \otimes_{K[X]} (K[X]/ (X-\xi)) \cong \tilde I_\gamma/(X-\xi)\tilde I_\gamma \cong I_{\xi,\gamma}.
\end{equation}

Since $K[X]$ is a principal ideal domain, $\tilde I_\gamma$ is a free $K[X]$-module.
Choose a $K[X]$-basis $\tilde\cB'$ for $\tilde I_\gamma$.
Thanks to~\eqref{eqn:sp-isom}, the image of $\tilde\cB'$ in $I_{\xi,\gamma}$ forms a $K$-basis for $I_{\xi,\gamma}$.
It follows that the image of $\cB \sqcup \tilde\cB'$ in $R_{\xi,\gamma}$ forms a $K$-basis for $R_{\xi,\gamma}$.

Consider the $K[X]$-submodule $M$ of $\tilde R_\gamma$ spanned by~$\cB \sqcup \tilde\cB'$.
The map $M \to R_{\xi,\gamma}$ 
(given by $X \mapsto \xi$) maps $\cB \sqcup \cB'$ to itself, hence is surjective.
Therefore, 
\[ \dim_K(R_{\xi,\gamma}) \leq \rk_{K[X]}(M) \leq \#\cB + \#\tilde\cB' = \dim_K(R_{\xi,\gamma}). \]
Therefore, these inequalities must be equalities, and~$\cB \sqcup \tilde\cB'$ is a $K[X]$-basis of~$M$.

Consequently, $\cB \sqcup \tilde\cB'$ is a $K(X)$-basis for $M \otimes_{K[X]} K(X) \subset R_\gamma$.
But $\dim_K(R_{\xi,\gamma}) = \dim_{K(X)}(R_\gamma)$ (since both are spanned by the same set of monomials).
Hence, $\cB \sqcup \tilde\cB'$ is a $K(X)$-basis for $R_\gamma$.
Also, $\tilde\cB'$ is a $K(X)$-basis for $\tilde I_\gamma \otimes_{K[X]} K(X) = I_\gamma$.
Therefore, the image of $\cB$ in $R_\gamma/I_\gamma$ is a $K(X)$-basis for $R_\gamma$.
\end{proof}

\begin{corollary} \label{hilbert-function-specialisation}
In the setting of \cref{lift-basis},
$HF_I(\gamma) = HF_{I_\xi}(\gamma)$.
\end{corollary}

This corollary could be restated as: if $\tilde V$ denotes the closed subscheme of $\AAA^1 \times \PP^{\numfuns-1}$ cut out by $I \cap K[X][Y_1,\dotsc,Y_\numfuns]$, where $I$ is a homogeneous ideal of $K(X)[Y_1,\dotsc,Y_\numfuns]$, then $\tilde V \to \AAA^1$ is very flat.

\begin{proof}[Proof of \cref{andre-bombieri-HF}]
Let $I_\xi \subset \Qbar[Y_1,\dotsc,Y_\numfuns]$ be the homogeneous ideal generated by the relations between $y_1,\dotsc,y_\numfuns$ which are trivial at~$\xi$.
In other words, $I_\xi$ is the specialisation of~$I \cap \Qbar[X][Y_1,\dotsc,Y_\numfuns]$ under $X \mapsto \xi$.

Using \cref{hilbert-function-specialisation}, set
\[ \nummonomials = HF_I(\delta+\gamma) = HF_{I_\xi}(\delta+\gamma), \qquad \kappa = HF_I(\gamma) = HF_{I_\xi}(\gamma). \]
Choose monomials:
\begin{enumerate}
\item $m_1,\dotsc,m_\nummonomials$ in $Y_1,\dotsc,Y_\numfuns$ of degree $\delta+\gamma$ whose images in $\Qbar[Y_1,\dotsc,Y_\numfuns]/I_\xi$ form a $\Qbar(X)$-basis for $\Qbar[Y_1,\dotsc,Y_\numfuns]_{\delta+\gamma}/I_{\xi,\delta+\gamma}$;
\item $m'_1,\dotsc,m'_\kappa$ in $Y_1,\dotsc,Y_\numfuns$ of degree $\gamma$ whose images in $\Qbar[Y_1,\dotsc,Y_\numfuns]/I_\xi$ form a $\Qbar$-basis for $\Qbar[Y_1,\dotsc,Y_\numfuns]_{\gamma}/I_{\xi,\gamma}$.
\end{enumerate}

By \cref{lift-basis}, the images of $m_1, \dotsc, m_\nummonomials$ in $\Qbar(X)[Y_1,\dotsc,Y_\numfuns]_{\delta+\gamma}/I_{\delta+\gamma}$ form a $\Qbar(X)$-basis.
It  follows that the $\Qbar(X)$-linear span of $\{ m_i(y_1,\dotsc,y_\numfuns) : 1 \leq i \leq \nummonomials \}$ is equal to the $\Qbar(X)$-linear span of all monomials of degree~$\delta+\gamma$ in $y_1,\dotsc,y_\numfuns$, hence is closed under $d/dX$.

Let $\relation$ be a super-strongly non-trivial $(r_v)$-global relation of degree~$\delta$ between $y_1,\dotsc,y_\numfuns$ at~$\xi$.
For each $i=1,\dotsc,\kappa$, we have $m_i'\relation \in \Qbar[Y_1,\dotsc,Y_\numfuns]_{\delta+\gamma}$.
Hence there exists some $\relation_i'$ in the $\Qbar$-span of $m_1,\dotsc,m_\nummonomials$ such that $\relation_i' - m_i'\relation_i \in I_{\xi,\delta+\gamma}$.
Then $\relation_1', \dotsc, \relation_\kappa'$ are $(r_v)$-global relations of degree~$\delta+\gamma$ between $y_1,\dotsc,y_\numfuns$ at~$\xi$.


Since $\relation$ is super-strongly non-trivial and $m_1',\dotsc,m_\kappa'$ are $\Qbar$-linearly independent mod~$I_\xi$, it follows that $\relation_1',\dotsc,\relation_\kappa'$ are $\Qbar$-linearly independent mod~$I_\xi$.
Therefore, we may apply \cref{andre-bombieri} to the relations $\relation_1', \dotsc, \relation_\kappa'$ and monomials $m_1,\dotsc,m_\nummonomials$.
\end{proof}

\subsection{Bounds involving certain binomial coefficients}

When applying \cref{andre-bombieri-HF}, one wants to choose $\gamma$ so that
$HF_I(\delta+\gamma)^3/HF_I(\gamma)^2$ is as small as possible.
When using a bound for $HF_I$ of the form~\eqref{eqn:HF-general-estimate}, this means we want to minimise
\[ \binom{\delta+\gamma+\eta}{\eta}^3 \Big/ \binom{\gamma+\eta}{\eta}^2. \]
When $\delta$ is large compared with $\eta$, the optimal choice is close to $\gamma = 2\delta - \lfloor \frac{1}{2}\eta \rfloor$.
The following lemma bounds the necessary binomial coefficients for this value of~$\gamma$.
This lemma is the source of the constants $\refC{hxi-strong-mu-kappa}(\eta)$ and~$\refC{hxi-strong-mu-mult}(\eta)$ defined in \cref{andre-bombieri-HF-bounds}.

\begin{lemma} \label{binom-delta-gamma-bounds}
Let $\eta,\delta \in \ZZ_{\geq1}$ with $\delta \geq \frac{1}{4}(\eta+2)$.
Set $\gamma = 2\delta - \lceil \frac{1}{2}\eta \rceil$ and
\[ \alpha(\eta,\delta) = \binom{\delta+\gamma+\eta}{\eta}, \quad \beta(\eta,\delta) = \binom{\gamma+\eta}{\eta}. \] 
Then:
\begin{enumerate}[(a)]
\item $\alpha(\eta,\delta) \leq \refC{hxi-strong-mu-mult}(\eta) \delta^{\eta}$;
\vskip 0.4em
\item $\alpha(\eta,\delta) / \beta(\eta,\delta) \leq \refC{hxi-strong-mu-kappa}(\eta)$.
\end{enumerate}
\end{lemma}

\begin{proof}
Note that $\delta \geq \frac{1}{4}(\eta+2)$ is equivalent to $\gamma \geq 1$.
Write $\ell = \lfloor \tfrac{1}{2}\eta \rfloor$.
Then
\[ \alpha(\eta,\delta) = \binom{3\delta + \ell}{\eta}, \quad \beta(\eta,\delta) = \binom{2\delta + \ell}{\eta}. \]

\subsubsection*{Proof of (a) when $\eta$ is odd}
If $\eta$ is odd, then $\eta=2\ell+1$ and
\begin{align*}
    \alpha(\eta,\delta)
  & = \frac{1}{\eta!}(3\delta+\ell)(3\delta+\ell-1) \dotsm (3\delta-\ell)
    \leq \frac{1}{\eta!}(3\delta)^\eta
\end{align*}
by the AM-GM inequality.

\subsubsection*{Proof of (a) when $\eta$ is even}
When $\eta$ is even, replacing $\eta$ by $\eta-1$ would reduce $\ell$ by~$1$, so
\begin{equation} \label{eqn:alpha-odd-reduce-n}
    \alpha(\eta,\delta) = \frac{3\delta+\ell}{\eta} \binom{3\delta+\ell-1}{\eta-1} = \frac{3\delta+\ell}{\eta} \alpha(\eta-1,\delta).
\end{equation}
Using (a) for $\eta-1$, which is odd, it follows that
\[ \alpha(\eta,\delta) \leq \frac{3\delta+\ell}{\eta} \cdot \frac{3^{\eta-1}\delta^{\eta-1}}{(\eta-1)!}
   = \Bigl( 1+\frac{\ell}{3\delta} \Bigr) \frac{3^\eta\delta^\eta}{\eta!}. \]
Since $\ell < 2\delta$, we have $1+\ell/3\delta < 5/3$, so this proves (a) when $\eta$ is even.

\subsubsection*{Proof of (b) when $\eta$ is odd}
If $\eta$ is odd, then
\begin{align*}
    \frac{\alpha(\eta,\delta)}{\beta(\eta,\delta)}
   = \prod_{i=-\ell}^{\ell} \frac{3\delta+i}{2\delta+i}
   = \prod_{i=-\ell}^{\ell} \Bigl( \frac{3}{2} - \frac{1}{2} \cdot \frac{i}{2\delta+i} \Bigr)
    = \frac{3}{2} \prod_{j=1}^{\ell} \Bigl( \frac{9}{4} + \frac{5}{4} \cdot \frac{j^2}{4\delta^2-j^2} \Bigr)
\end{align*}
(for the last equality, split out the $i=0$ factor, and combine the other factors in pairs $i=\pm j$).
For fixed $j$, as $\delta$ increases, the factor $\frac{9}{4} + \frac{5}{4} \cdot \frac{j^2}{4\delta^2-j^2}$ decreases (bearing in mind that $2\delta > \ell \geq j$).
Hence, for fixed odd~$\eta$, $\alpha(\eta,\delta)/\beta(\eta,\delta)$  is a decreasing function of~$\delta$.

Write $\eta=4k-1$ or $4k+1$ with $k \in \ZZ$.
In either of these cases, the condition that $\delta \geq \frac{1}{4}(\eta+2)$ is equivalent to $\delta \geq k+1$.
Hence, and comparing with the definition of $\refC{hxi-strong-mu-kappa}(\eta)$, we obtain that
\[ \frac{\alpha(\eta,\delta)}{\beta(\eta,\delta)} \leq \frac{\alpha(\eta,k+1)}{\beta(\eta,k+1)} = \refC{hxi-strong-mu-kappa}(\eta). \]

\subsubsection*{Proof of (b) when $\eta$ is even}
If $\eta$ is even and fixed, then $\alpha(\eta,\delta)/\beta(\eta,\delta)$ is an increasing function of~$\delta$, so the same argument as in the odd~$\eta$ case does not work.

Instead, we use \eqref{eqn:alpha-odd-reduce-n}, together with the similar
\begin{align*}
    \beta(\eta,\delta)  & = \frac{2\delta+\ell}{\eta} \binom{2\delta+\ell-1}{\eta-1} = \frac{2\delta+\ell}{\eta} \beta(\eta-1,\delta). 
\end{align*}
Hence,
\[ \frac{\alpha(\eta,\delta)}{\beta(\eta,\delta)} = \frac{3\delta+\ell}{2\delta+\ell} \cdot \frac{\alpha(\eta-1,\delta)}{\beta(\eta-1,\delta)} \leq \frac{3}{2} \cdot \frac{\alpha(\eta-1,\delta)}{\beta(\eta-1,\delta)}. \]
Thus, we may deduce (b) for an even value~$\eta$ from (b) for the odd value~$\eta-1$.
\end{proof}

\subsection{General bounds for Hilbert functions} \label{ssec:ab-super-strong-end}

The following lemma states an upper bound for the Hilbert functions of homogeneous prime ideals due to Chardin, and converts a lower bound due to Nesterenko into the form~\eqref{eqn:HF-general-estimate}.

\begin{lemma} \label{HF-bounds}
Let $I \subset K[Y_1,\dotsc,Y_\numfuns]$ be a homogeneous prime ideal, defining a projective variety $V \subset \PP^{\numfuns-1}$.
Let $\eta = \dim(V)$ and let $D = \deg(V)$.
Then we have
\[ HF_I(\gamma) \leq D\binom{\gamma+\eta}{\eta} \text{ for all } \gamma \geq 0 \]
and
\[ HF_I(\gamma) \geq \Bigl( 1-\frac{D}{\gamma} \Bigr)^\eta D \binom{\gamma+\eta}{\eta} \text{ for all } \gamma > D. \]
\end{lemma}

\begin{proof}
The upper bound is from \cite{Cha89}.

For the lower bound, we start from \cite[\S 6, Prop.~1]{Nes84}:
\begin{align*}
    HF_I(\gamma)
   \geq \binom{\gamma+\eta+1}{\eta+1} - \binom{\eta - D + \gamma+1}{\eta+1}
   \geq D \binom{\gamma+\eta-D}{\eta}.
\end{align*}
Using that $\gamma > D$, the lower bound stated in the lemma follows because
\begin{align*}
\binom{\gamma+\eta-D}{\eta}
   = \binom{\gamma+\eta}{\eta} \cdot \prod_{j=1}^\eta \frac{\gamma+j-D}{\gamma+j}
    \geq \binom{\gamma+\eta}{\eta} \cdot \biggl( 1-\frac{D}{\gamma} \biggr)^\eta.
\\[\dimexpr-\baselineskip+\dp\strutbox]
&\qedhere
\end{align*}
\end{proof}

\begin{corollary} \label{HF-bounds-2}
If $\gamma \geq 2D$, then
\[ 2^{-\eta} D \binom{\gamma+\eta}{\eta} \leq HF_I(\gamma) \leq D \binom{\gamma+\eta}{\eta}. \]
\end{corollary}

\begin{proof}
Since $\gamma \geq 2D$, we have $\tfrac{1}{2} \leq 1-D/\gamma$.
Hence this follows from \cref{HF-bounds}.
\end{proof}

\section{Application to unlikely intersections with lines in torus} \label{sec:torus-application}

In this section, we prove \cref{lines-in-torus-height-bound-intro}, that is, an explicit height bound for unlikely intersections in tori (under the restrictions discussed in section~\ref{torus-bound-restrictions}).
In order to state the bound, for number fields~$L$, we write
\[ \tau(L) = \# \{ \text{complex places of~$L$} \} = \tfrac{1}{2} \bigl( [L:\QQ] - \# \Hom(L,\RR) \bigr). \]

\begin{theorem} \label{lines-in-torus-height-bound}
Let $n \in \ZZ_{\geq2}$.
Let $G = \GG_m^n$.
Let $C \subset G$ be a closed algebraic curve such that the Zariski closure~$\ov C$ of~$C$ in~$\AAA^n$ is a line defined over a number field~$K$.
Suppose that $C$ contains the point $s_0 := (1,1,\dotsc,1)$, and that $C$ is not contained in a proper algebraic subgroup of~$G$.

Let $x \colon \ov C \to \AAA^1$ be an isomorphism of algebraic curves satisfying $x(s_0)=1$.
Let
\[ \Xi = \{ x(s) : s \in \ov C(\ov K) \setminus C(\ov K) \}. \]
Let $H$ be an upper bound for $H(\Xi)$.

Then, for every point~$s \in C(\ov K) \cap \Sigma_2$, letting $\delta = \max\{1,\tau(K(s))\}$, we have:
\begin{enumerate}[(i)]
\item for all $\delta \in \ZZ_{\geq1}$, letting $\nummonomials = \binom{n+\delta}{\delta}$,
\begin{equation*}
h(x(s)) < n \bigl( 8\nummonomials^3 - 3.8\nummonomials^2 \bigr) \bigl( \log(\delta) + 2 \bigr) (H+1);
\end{equation*}

\item if $\delta \geq \tfrac{1}{4}(n+2)$, then
\begin{align*}
    h(x(s)) < {}
  & 8n \refC{hxi-strong-mu-kappa}(n)^2 \refC{hxi-strong-mu-mult}(n)  \delta^n \big( \log(3\delta) + 2 \bigr)(H+1),
\end{align*}
where $\refC{hxi-strong-mu-kappa}(n)$ and $\refC{hxi-strong-mu-mult}(n)$ were defined in \cref{andre-bombieri-HF-bounds}.
\end{enumerate}
\end{theorem}

While the bound in \cref{lines-in-torus-height-bound}(i) is valid for all~$\delta$ and~$n$, it is mainly useful for $\delta < \tfrac{1}{4}(n+2)$, because \cref{lines-in-torus-height-bound}(ii) is usually better when it applies.

Throughout the proof of \cref{lines-in-torus-height-bound}, we shall let $p_1,\dotsc,p_n \colon \ov C \to \AAA^1$ denote the projections onto the coordinate axes.
For each~$i$, we have $p_i = 1 + a_ix$ in~$\cO(\ov C)$, for some $a_i \in K^\times$.
The set $\Xi$ can be described more concretely as
\begin{equation} \label{eqn:Xi-ai}
\Xi = \{ x(s) : s \in \ov C(\ov K), p_i(s) = 0 \text{ for some } i \} = \{ -1/a_i : 1 \leq  i \leq n \}.
\end{equation}
Consequently,
\[ H(\Xi) = h(a_1) + \dotsb + h(a_n). \]

The proof of \cref{lines-in-torus-height-bound} is in two parts.
In section~\ref{ssec:tori-inexplicit}, we prove the existence of a bound for $h(x(s))$ which is polynomial in~$\delta$, but without computing the constants; the bulk of the work is to construct $(r_v)$-global relations at points in $C(\ov K) \cap \Sigma_2$.
Then, in section~\ref{ssec:tori-explicit}, we compute the explicit constants, by computing the input data of \cref{andre-bombieri} for conclusion~(i) and \cref{andre-bombieri-homog-alg-ind} for conclusion~(ii).

\subsection{Proof of inexplicit height bound} \label{ssec:tori-inexplicit}

\subsubsection{G-functions}

For each $1 \leq i \leq n$, let
\[ F_i = \log(1+a_i X) \in \powerseries{K}{X}. \]
This is the Taylor series of ${\log} \circ p_i$ around~$s_0$, in terms of the local parameter~$x$.
It was shown in section~\ref{ssec:g-functions-definition} that the $F_i$ are G-functions.

We shall apply \cref{andre-thm-E} to the set of G-functions
\[ \cG = \{ 1, F_1, \dotsc, F_n \}. \]
The $(r_v)$-global relations we construct will involve only the G-functions $F_1, \dotsc, F_n$; by adjoining the G-function~$1$, we obtain a set~$\cG$ whose $\Qbar(X)$-linear span is closed under $d/dX$, because, for each~$i$,
\begin{equation} \label{eqn:F_i-derivative}
\frac{dF_i}{dX} = \frac{a_i}{1+a_iX} \in \Qbar(X) \cdot 1.
\end{equation}

\subsubsection{Functional relations}


The functional relations are controlled using the Ax--Schanuel theorem.

\begin{lemma} \label{Fi-alg-indep}
The elements of~$\cG$ are homogeneously algebraically independent over~$\Qbar(X)$.
\end{lemma}

\begin{proof}
Since $C$ is not contained in any proper algebraic subgroup of~$G$ and $F_i(0) = \log(p_i(s_0)) = 0$ for each~$i$,  the power series $F_1,\dotsc,F_n$ are $\QQ$-linearly independent and have no constant terms.
Hence, by the Ax--Schanuel theorem \cite[Thm.~1]{Ax71},
\[ \trdeg_{\Qbar}(F_1,\dotsc,F_n,\exp F_1,\dotsc,\exp F_n) \geq n+\rk \bigl( \rd F_i/\rd X : 1 \leq i \leq n \bigr) = n+1. \]
Since $\exp F_1, \dotsc, \exp F_n \in \Qbar(X)$, it follows that
\[ \trdeg_{\Qbar(X)}(F_1,\dotsc,F_n) \geq n. \]
In other words, $F_1, \dotsc, F_n$ are algebraically independent over~$\Qbar(X)$.

Any homogeneous $\Qbar(X)$-polynomial relation between $1,F_1,\dotsc,F_n$ would give rise to a (not necessarily homogeneous) $\Qbar(X)$-relation between $F_1,\dotsc,F_n$.
Hence $1,F_1,\dotsc,F_n$ are homogeneously algebraically independent over~$\Qbar(X)$.
\end{proof}

\subsubsection{Acceptable system of radii} \label{sssec:system-of-radii}



Let $(r_v)$ be a system of radii over~$K$ such that the $K_v$-analytic spaces
\[ U_v = (x^\van)^{-1}(D(0,r_v,K_v)) \cap C^\van \]
satisfy:
\begin{enumerate}[(i)]
\item for each place~$v$ and each $i=1,\dotsc,n$, we have $U_v \subset (p_i^\van)^{-1}(D(1,1,K_v))$;
\item $(r_v)$ is acceptable for the set of G-functions~$\cG$.
\end{enumerate}
Since $p_i$ is a morphism of algebraic varieties over~$K$, there exists a system of radii satisfying (i) and~(ii) in which almost all~$r_v$ are equal to~$1$.
Hence, after potentially shrinking the $r_v$, we may choose them to also satisfy~(iii).


\subsubsection{Construction of relations}

In this subsection, we construct a global relation between the elements of~$\cG$ at~$x(s)$, whenever $s \in C(\ov K) \cap \Sigma_2$.

\begin{remark} \label{arch-vs-non-arch-relations}
The basic strategy, as in previous applications of the G-functions method involving relations at non-archimedean places (such as \cite{DO:LGO-PEL}), is to find one $\hat w$-adic relation~$P_{\hat w}$ for each archimedean place~$w$ of~$K(s)$, and one relation~$P_{na}$ which works at all non-archimedean places.
In the present setting, where the local relations are linear, multiplying these together would lead to an $(r_v)$-global relation $P_{na} \cdot \prod_{\hat w \text{ arch}} P_{\hat w}$ of degree at most $[K(s):\QQ] + 1$.

We improve on this slightly but significantly.
It happens that our relations for complex places also work at all non-archimedean and real places.
Hence, if there is at least one $(r_v)$-relevant complex place, we don't need $P_{na}$, and $\prod_{\hat w \text{ complex}} P_{\hat w}$ is already $(r_v)$-global.
On the other hand, if there are no $(r_v)$-relevant complex places, then $P_{na}$ works at the real as well as the non-archimedean places, so is itself an $(r_v)$-global relation.
Thus we get an $(r_v)$-global relation of degree at most
\[ \max \{ \tau(K(s)), 1 \} \leq \max \{ \tfrac{1}{2}[K(s):\QQ], 1 \}. \]
\end{remark}

If $s$ lies in a proper algebraic subgroup of~$G$, then there exists a non-trivial homomorphism $\phi \colon G \to \GG_m$ such that $s \in \ker(\phi)$.
We use such a homomorphism to construct a relation between values of $F_1,\dotsc,F_n$ as follows.

\begin{definition}
Let $\phi \colon G \to \GG_m$ be a homomorphism.
Write it in the form
\[ \phi(X_1,\dotsc,X_n) = X_1^{b_1} \dotsm X_n^{b_n}, \]
where $b_1,\dotsc,b_n \in \ZZ$.
Define the \defterm{linear relation attached to~$\phi$} to be
\[ P_\phi(Y_1,\dotsc,Y_n) = b_1Y_1 + \dotsb + b_nY_n \in \QQ[Y_1,\dotsc,Y_n]. \]
\end{definition}

\begin{remark} \label{linear-relations-non-zero}
In the above definition, note that if $\phi \neq 1$, then $P_\phi \neq 0$.
\end{remark}

The non-archimedean and real places are easier to deal with.
Indeed, we can get a $\hat v$-adic relation for $(r_v)$-relevant non-archimedean and real places~$\hat v$ whenever $\underline p(s)$ lies in any proper algebraic subgroup of~$G$ (even of codimension~$1$).

Let $\hat K$ be a finite extension of~$K$.
For each place~$\hat v$ of~$\hat K$, let $\log^\vhatan$ denote the $\hat K_{\hat v}$-analytic function on the disc $D(1,1,\hat K_{\hat v})$ whose Taylor series around~$1$ is $\log(1+(X-1))$.

\begin{lemma} \label{relation-linear-non-arch}
Let $s \in C(\hat K)$.
Let $\phi \colon G \to \GG_m$ be a homomorphism such that $\phi(s)=1$.
Then, for every $(r_v)$-relevant non-archimedean or real place $\hat v$ of~$\hat K$, the linear relation~$P_\phi$ is a $\hat v$-adic relation between $F_1, \dotsc, F_n$ at~$x(s)$.
\end{lemma}

\begin{proof}
By a well-known result in $p$-adic analysis, if $\hat v$ is non-archimedean, $\log^\vhatan$ is a rigid group homomorphism $D(1,1,\hat K_{\hat v}) \to \GG_{a,\hat K}^\vhatan$ \cite[V.4.1]{Rob00}.
Similarly, if $\hat v$ is real, then $\log^\vhatan$ is the restriction to $D(1,1,\RR)$ of the group homomorphism $\log \colon (\RR_{>0}, \cdot) \to (\RR, +)$.

Therefore, writing $\phi(X_1,\dotsc,X_n) = X_1^{b_1} \dotsm X_n^{b_n}$, we have
\begin{align*}
  & P_\phi^\van \bigl( F_1^\van(x(s)^\vhatan), \dotsc, F_n^\van(x(s)^\vhatan) \bigr)
\\& = b_1 \log^\vhatan(p_1(s)^\vhatan) \, + \, \dotsb \, + \, b_n \log^\vhatan(p_n(s)^\vhatan)
\\& = \log^\vhatan \bigl( (p_1(s)^{b_1} \dotsm p_n(s)^{b_n})^\vhatan \bigr)
\\& = \log^\vhatan(1) = 0.
\qedhere
\end{align*}
\end{proof}

On the other hand, for a complex place~$\hat w$, $\log^\whatan$ is not a restriction of a group homomorphism, but only a homomorphism up to multiples of~$2\pi i$ (the complex periods of~$\GG_m$).
In other words,
\[ \log^\whatan(zz') \in \log^\whatan(z) + \log^\whatan(z') + 2\pi i\ZZ. \]
Consequently, the same argument as in the proof of \cref{relation-linear-non-arch} yields only an \emph{inhomogeneous} $\hat w$-adic relation, with a multiple of~$2\pi i$ on the right hand side:

\begin{lemma} \label{relation-linear-arch-inhomog}
Let $s \in C(\hat K)$. Let $\phi \colon G \to \GG_m$ be a homomorphism such that $\phi(s)=1$.
Then, for every $(r_v)$-relevant complex place $\hat w$ of~$\hat K$, we have
\[ P_\phi^\wan(F_1^\wan(x(s)^\whatan), \dotsc, F_n^\wan(x(s)^\whatan)) \in 2\pi i\ZZ. \]
\end{lemma}

Given \emph{two} multiplicative relations satisfied by a point~$s \in G(\hat K)$, we can take a linear combination of the inhomogeneous relations given by \cref{relation-linear-arch-inhomog} in order to obtain a (homogeneous) $\hat w$-adic relation.
In order to compare with the relations at non-archimedean places, as mentioned in \cref{arch-vs-non-arch-relations}, it is convenient to describe this in terms of taking a combination of two homomorphisms which annihilate~$s$.

\begin{lemma} \label{relation-linear-arch}
Let $s \in C(\hat K) \cap \Sigma_2$.
Let $\hat w$ be an $(r_v)$-relevant complex place $\hat w$ of~$\hat K$.
Then there exists a non-constant homomorphism $\phi_{\hat w} \colon G \to \GG_m$ such that $\phi_{\hat w}(s)=1$ and the linear relation~$P_{\phi_{\hat w}}$ is a $\hat w$-adic relation between $F_1,\dotsc,F_n$ at~$x(s)$.
\end{lemma}

\begin{proof}
Since $s \in \Sigma_2$, there exist at least two $\ZZ$-linearly independent homomorphisms $\phi_1,\phi_2 \colon G \to \GG_m$ such that $\phi_1(s)=\phi_2(s)=1$.
By \cref{relation-linear-arch-inhomog}, there exist integers $r_1,r_2$ (which depend on the place~$\hat w$) such that
\[ P_{\phi_j}^\wan(F_1^\wan(x(s)^\whatan), \dotsc, F_n^\wan(x(s)^\whatan)) = 2\pi ir_j \]
for each $j=1,2$.

Choose integers $a_1,a_2$, not both zero, such that $a_1r_1 + a_2r_2=0$.
Defining
\[ \phi_{\hat w} = \phi_1^{a_1} \phi_2^{a_2} \colon G \to \GG_m, \]
we obtain a homomorphism of tori satisfying $\phi_{\hat w}(s)=1$ and $P_{\phi_{\hat w}} = a_1 P_{\phi_1} + a_2 P_{\phi_2}$.
Then $P_{\phi_{\hat w}}$ is a $\hat w$-adic relation between $F_1,\dotsc,F_n$ at~$x(s)$.
\end{proof}

\begin{proposition} \label{global-relation}
Let $s \in C(\ov K) \cap \Sigma_2$.
Then there exists a non-zero $(r_v)$-global relation between $F_1,\dotsc,F_n$ at~$x(s)$, of degree at most~$\max \{ \tau(K(s)), 1 \}$.
\end{proposition}

\begin{proof}
Let $\hat K = K(s)$.

First, suppose that there exists at least one $(r_v)$-relevant complex place of~$\hat K$ for~$x(s)$.
In this case, let $\phi_{\hat w}$ be the homomorphism given by \cref{relation-linear-arch}, for each $(r_v)$-relevant complex place~$\hat w$.
Then let
\[ P = \prod_{\hat w} P_{\phi_{\hat w}}, \]
where the product is over all $(r_v)$-relevant complex places of~$\hat K$.
Since each $P_{\phi_{\hat w}}$ is linear, $\deg(P) \leq \tau(\hat K)$.
By \cref{linear-relations-non-zero}, $P \neq 0$.
Thanks to \cref{relation-linear-arch}, $P$ is a $\hat w$-adic relation for each $(r_v)$-relevant complex place~$\hat w$.
Furthermore, if we pick one complex place~$\hat w$, then, by \cref{relation-linear-non-arch}, $P_{\phi_{\hat w}}$ (hence~$P$) is a $\hat v$-adic relation for every $(r_v)$-relevant non-archimedean or real place.
Thus, $P$ is an $(r_v)$-global relation.

Now, suppose that there are no $(r_v)$-relevant complex places for~$x(s)$.
In this case, choose a non-constant homomorphism $\phi \colon G \to \GG_m$ such that $\phi(s) = 1$.
Let $P$ be the linear relation $P_\phi$.
This is an $(r_v)$-global relation by \cref{relation-linear-non-arch}.
\end{proof}

\subsubsection{Application of principle of global relations}

Consider a point $s \in C(\ov K) \cap \Sigma_2$.
After excluding finitely many points, we may assume that $x(s)$ is an ordinary point of the system of differential equations satisfied by~$\cG$.
(In fact, by \cref{sin-Lambda} below, $x(s)$ is ordinary for every $s \in C(\ov K)$.)
Since the elements of~$\cG$ are homogeneously algebraically independent, the $(r_v)$-global relation constructed in \cref{global-relation} is non-trivial.
Applying \cref{andre-thm-E}, using the $(r_v)$-global relation constructed in \cref{global-relation}, we obtain a bound of the form
\[ h(x(s)) < \newC* \delta^{\newC*}. \]

\subsection{Explicit height bound in tori} \label{ssec:tori-explicit}

\subsubsection{Input data}

Let $\Lambda$ denote the differential operator $d/dX - \Gamma$, where
\[ \Gamma =
\begin{pmatrix}
   0                 & 0      & \cdots & 0
\\ \frac{a_1}{1+a_1X} & 0      & \cdots & 0 
\\ \vdots            & \vdots &        & \vdots
\\ \frac{a_n}{1+a_nX} & 0      & \cdots & 0 
\end{pmatrix}. \]
According to equation~\eqref{eqn:F_i-derivative}, $(1,F_1,\dotsc,F_n)^t$ is a solution of~$\Lambda$.

\begin{lemma} \label{sin-Lambda}
The singularities of~$\Lambda$ are the elements of~$\Xi$. They are all non-apparent.
\end{lemma}

\begin{proof}
By~\eqref{eqn:Xi-ai}, $\Xi$ is the set of poles of entries of~$\Gamma$, that is, the singularities of~$\Lambda$.

For $1 \leq i \leq n$, the residue matrix of~$\Gamma$ at~$-1/a_i$ has entry $1$ at position~$(i+1,1)$ and $0$ everywhere else.
This matrix is not diagonalisable, so $-1/a_i$ is not an apparent singularity \cite[Rmk.~1]{BM15}.
\end{proof}

\begin{lemma} \label{bounds-d-module-logarithm}
The generic radii of solvability, global radius and size of the differential operator~$\Lambda$ satisfy:
\begin{enumerate}[(a)]
\item for every place $v$ of~$K$, $R_v(\Lambda) \geq 1$;
\item $\rho(\Lambda) = 0$;
\item $\sigma(\Lambda) \leq 1$.
\end{enumerate}
\end{lemma}

\begin{proof}
Let $\cM$ denote the $\partial$-module structure on $K(X)^{n+1}$ induced by~$\Lambda$. 
By definition, $R_v(\Lambda)=R_v(\cM)$, $\rho(\Lambda)=\rho(\cM)$ and $\sigma(\Lambda)=\sigma(\cM)$.
Thanks to \eqref{eqn:F_i-derivative}, the $\partial$-module~$\cM$ sits in an exact sequence
\[ 0 \to K(X) \to \cM \to K(X)^n \to 0, \]
where $K(X)$ carries its standard $\partial$-module structure.

We have $R_v(K(X)^n) = 1$ for all~$v$, and $\rho(K(X)^n) = \sigma(K(X)^n) = 0$.
Consequently, conclusion~(a) follows from \cite[IV, 2.5, Prop.~1]{And89}.
Conclusion~(b) is an immediate consequence of~(a).

For conclusion~(c), by \cite[IV, 4.1, Prop.\ and Lemma~2(a)]{And89},
we obtain
\[ \sigma(\cM) \leq 1 + 2\sigma(K(X)) + 2\sigma(K(X)^n) + 2\sigma((K(X)^n)^*) = 1.
\qedhere \]
\end{proof}

\begin{lemma} \label{bounds-series-logarithm}
The radii of convergence, global radius and size of the sequence of power series $1,F_1,\dotsc,F_n$ satisfy:
\begin{enumerate}[(a)]
\item for every place $v$ of~$K$, $R_v(1,F_1,\dotsc,F_n) \geq \sigma_v(\Xi)$;
\item $\rho(1,F_1,\dotsc,F_n) \leq \sigma(\Xi)$;
\item $\sigma(1,F_1,\dotsc,F_n) \leq \sigma(\Xi) + 1$.
\end{enumerate}
\end{lemma}

\begin{proof}
The power series $\log(1+X) \in \QQ[X]$ has radius of convergence~$1$ at every place.
Therefore, $F_i(X) = \log(1+a_iX)$ has $v$-adic radius of convergence $\abs{a_i}_v^{-1}$, so
\[ R_v(1,F_1,\dotsc,F_n) = \min \{ \infty, \abs{a_1}_v^{-1}, \dotsc, \abs{a_n}_v^{-1} \} \geq \sigma_v(\Xi). \]

Conclusion~(b) follows immediately from~(a).

For~(c), let $h_i = \rd F_i/\rd X = a_i/(1+a_iX)$.
Since the $h_i$ are rational functions, by \cite[I, 4.1, Lemma]{And89}, we have
\[ \sigma(1,h_1,\dotsc,h_n) = \sigma(\mathrm{poles}(1,h_1,\dotsc,h_n)) = \sigma(\Xi). \]
Since $F_i = \int_0 h_i$, by \cite[I, 1.4, Lemma~2(c)]{And89}, we have
\[ \sigma(1,F_1,\dotsc,F_n) \leq \sigma(1,h_1,\dotsc,h_n) + 1.
\qedhere \]
\end{proof}

\begin{lemma} \label{radii-simple-neighbourhoods}
The system of radii defined by $r_v = \sigma_v(\Xi)$ satisfies conditions (i) and~(ii) of section~\ref{sssec:system-of-radii}.
\end{lemma}

\begin{proof}

For each place $v$ of~$K$, we have
\[ \abs{p_i(x^{-1}(z)) - 1}_v = \abs{(1+a_iz) - 1}_v = \abs{a_iz}_v < \abs{a_i}_v r_v \leq \sigma_v(\Xi)^{-1} \sigma_v(\Xi) = 1 \]
for all $z \in D(0,r_v,K_v)$.
This proves condition~(i).

\Cref{bounds-series-logarithm}(a) and the fact that $\sigma_v(\Xi)$ is $1$ for all but finitely many~$v$ ensure that condition~(ii) holds.
\end{proof}

Thanks to the above lemmas, we may take the following values for the input data listed in section~\ref{ssec:input-data}:
\begin{enumerate}[(i)]
\item $\numfuns=n+1 \geq 2$;
\item $s = \#\Xi = n$ (\cref{sin-Lambda});
\item $H = H$;
\item $\rho = {\sum_v}' \log^+(1/\sigma_v(\Xi)) = \sigma(\Xi) \leq H(\Xi) \leq H$ (\cref{radii-simple-neighbourhoods});
\item $\sigma_y = \sigma(\Xi) + 1 \leq H+1$ (\cref{bounds-series-logarithm}(c));
\item $\sigma_\Lambda = 1$ (\cref{bounds-d-module-logarithm}(c)).
\end{enumerate}

\subsubsection{Calculation of explicit height bounds}

As in the statement of \cref{lines-in-torus-height-bound}, let $\delta = \max \{ 1, \tau(K(s)) \}$.

To obtain \cref{lines-in-torus-height-bound}(i), apply \cref{andre-bombieri} with $\kappa=1$ and with $\{m_1,\dotsc,m_\nummonomials\}$ being the set of all monomials of degree~$\delta$ in $n+1$ indeterminates, so that
\[ \nummonomials = \#\cM = \binom{n+\delta}{\delta}. \]
In this case (at least when~$n=3$ and $\delta=1$), $\nummonomials$ is small enough that the improvement which comes from using \eqref{eqn:hxi-general-bound} instead of \eqref{eqn:hxi-general-bound-simp} is worthwhile.
Substituting (i)--(vi), and $\kappa=1$, into \eqref{eqn:hxi-general-bound}, we obtain, for all $s \in C(\ov K) \cap \Sigma_2$,
\begin{align}
    h(x(s)) < {}
  & 8n\nummonomials^2(\nummonomials-\tfrac{1}{2}) \bigl( (\log(\delta)+1)(H+1) + H \bigr)
\notag
\\& + 6\nummonomials \log(2) + \frac{1}{2n}(\log(\delta)+1) + \frac{1}{n}H.
\label{eqn:line-in-torus-weak-full-bound}
\end{align}
When $n \geq 3$, the second line of~\eqref{eqn:line-in-torus-weak-full-bound} is at most
\[ 0.2n \bigl( \log(\delta)+2 \bigr) (H+1). \]
Hence \eqref{eqn:line-in-torus-weak-full-bound} implies \cref{lines-in-torus-height-bound}(i).

When $n=2$, $\Sigma_2$ consists only of torsion points of~$\GG_m^2$, so it is easily seen that every point of $C(\ov K) \cap \Sigma_2$ satisfies
\[ h(x(s)) \leq H+1 \]
and hence satisfies \cref{lines-in-torus-height-bound}(i).

To obtain \cref{lines-in-torus-height-bound}(ii), we use \cref{andre-bombieri-homog-alg-ind}.
(Note that, in this case, $\delta \geq \tfrac{1}{4}(n+2) \geq \tfrac{1}{4}(\numfuns+1)$.)
Substituting (i)--(vi) into~\eqref{eqn:hxi-strong-bound}, and noting that $\theta_\Lambda = \max\{1, H\}$, gives \cref{lines-in-torus-height-bound}(ii).

To deduce \Cref{lines-in-Gm3-height-bound-intro}, substitute $n=3$, $\nummonomials=4$ into \cref{lines-in-torus-height-bound}(i) for the case $\delta=1$, and note that this includes the cases where $K(s)$ is totally real or $[K(s):\QQ] \leq 3$.
For the case $\delta \geq 2 \geq \tfrac{1}{4}(3+2)$, substitute $n=3$, $\refC{hxi-strong-mu-kappa}(3) = \frac{7}{2}$, $\refC{hxi-strong-mu-mult}(3) = \frac{9}{2}$ into \cref{lines-in-torus-height-bound}(ii) and use that $\tau(K(s)) \leq \tfrac{1}{2}[K(s):\QQ]$.

\subsection{Comparison with Habegger's explicit bound} \label{ssec:habegger}

In \cite{Hab17}, Habegger proved an explicit height bound for \cref{bmz} (that is, for unlikely and just likely intersections in tori).
The statement is as follows.

\begin{theorem} \label{habegger-bound} \cite[Thm.~11(iii) with $r=s=1$]{Hab17}
Let $n \geq 1$.
Let $C \subset \GG_m^n$ be an irreducible algebraic curve defined over~$\Qbar$.
Suppose that $C$ is not contained in any translate of a proper algebraic subgroup of~$\GG_m^n$.
Then every point $s \in C(\Qbar) \cap \Sigma_1$ satisfies
\[ h_2(s) \leq (2\deg(C))^{\refC{habegger-exponent}(n)} (h(C) + 1), \]
where $\newC{habegger-exponent}(n) = 600n^7(2n)^{n^2}$.
\end{theorem}

We write $h_2(s)$ for the height denoted by~$h(s)$ in \cite{Hab17} (this is different from the present paper's $h(s)$).
The height of the curve~$h(C)$ was defined by Philippon~\cite{Philippon}, and a summary of the definition can be found in \cite[\S 3.2]{Hab17}.
To compare \cref{habegger-bound} with \cite[Thm.~11]{Hab17}, note that, when $C$ is a curve, the set called $C^{oa,[1]}$ in \cite{Hab17} is equal to~$C$.

The notions of height used in \cref{lines-in-torus-height-bound,habegger-bound} may be compared as follows.

\begin{lemma} \label{height-comparison-point}
Let $\Xi = \{ a_1,\dotsc,a_n \} \subset K$.
Let $\ov C$ denote the line
\[ \ov C = \{ (1+a_1\xi, \dotsc, 1+a_n\xi) : \xi \in \AAA^1 \} \subset \AAA^n. \]

\begin{enumerate}
\item Every point $s = (1+a_1\xi, \dotsc, 1+a_n\xi) \in \ov C(\ov K)$ satisfies 
\[ \abs{h_2(s)-h(\xi)} \leq \sigma(\Xi) + \log(2\sqrt{n+1}). \]
\item Suppose that $a_i=1$ for some~$i$ (which may be achieved by scaling).
Then 
\[ \abs{h(\ov C) - \sigma(\Xi)} \leq (2n+1)\log(2) + 1. \]
\end{enumerate}
\end{lemma}

The details of \cref{height-comparison-point} are not important.
The point is that the error bounds in \cref{height-comparison-point} and inequality~\eqref{eqn:sigma-H-comparison} are tiny compared to the factor $2^{\refC{habegger-exponent}(n)}$ in \cref{habegger-bound}.
So we may ignore the different measures of height when comparing \cref{lines-in-torus-height-bound} with \cref{habegger-bound}.

When $C$ is a line, $\deg(C)=1$, so \cref{habegger-bound} gives
\[ h_2(s) \leq 2^{\refC{habegger-exponent}(n)} (h(C) + 1). \]
When $n=3$, we have $\refC{habegger-exponent}(3) = 600 \cdot 3^7 \cdot 6^9 \approx 1.3 \cdot 10^{13}$, so this gives roughly
\[ h_2(s) \leq 10^{4 \cdot 10^{12}} (h(C)+1). \]
When $[K(s):\QQ]$ is small (or when $K(s)$ is totally real), \cref{lines-in-Gm3-height-bound-intro} is much better than this.

To see how large $\delta$ should be for \cref{habegger-bound} to beat \cref{lines-in-torus-height-bound}, note that
\[ \log_2 \bigl( 8n^3 \refC{hxi-strong-mu-kappa}(n)^2 \refC{hxi-strong-mu-mult}(n) \bigr) < 14 + 3n \]
is tiny when compared with $\refC{habegger-exponent}(n)$.
Hence, the constants in \cref{lines-in-torus-height-bound} can essentially be ignored in the comparison.
Thus, the cross-over occurs when
\[ n \log_2(\delta) \approx \refC{habegger-exponent}(n), \]
that is, when 
\[ \delta \approx 2^{600n^6(2n)^{n^2}}. \]
For example, when $n=3$, \cref{lines-in-Gm3-height-bound-intro} is better than \cref{habegger-bound} for $\delta$ up to around $10^{10^{12}}$.

\begin{remark}
The comparison above is not entirely fair.
The enormous constant in \cref{habegger-bound} partly comes from the fact that \cite{Hab17} proves a bound for unlikely and just likely intersections with arbitrary subvarieties in~$\GG_m^n$, not just curves.
Further improvements are possible to \cite{Hab17} by specialising to lines, most significantly to \cite[Prop.~8.1]{Hab17}.
By optimising the calculations of \cite{Hab17} for the special case of lines, one can improve the factor $2^{\refC{habegger-exponent}(n)}$ to $2^{27} n^{17}$.
When $n=3$, this is roughly $10^{16}$.
This beats \cref{lines-in-Gm3-height-bound-intro} for $\delta$ greater than around $10\,000$.
\end{remark}

\bibliographystyle{amsalpha}
\bibliography{heights}

\end{document}